\documentclass[12pt, reqno]{amsart}
\usepackage{}
\usepackage{stmaryrd}
\usepackage{amsmath}
\usepackage{amsfonts}
\usepackage{amssymb}
\usepackage[all,cmtip]{xy}           
\usepackage{xspace}
\usepackage{bbding}
\usepackage{txfonts}
\usepackage[shortlabels]{enumitem}
\usepackage{cite}

\usepackage{tikz}
\usepackage{titletoc}

\usepackage{ifpdf}
\ifpdf
 \usepackage[colorlinks,final,backref=page,hyperindex]{hyperref}
\else
 \usepackage[colorlinks,final,backref=page,hyperindex,hypertex]{hyperref}
\fi
\usepackage[active]{srcltx}

\makeatletter

\begin{document}
\newcommand {\emptycomment}[1]{} 

\baselineskip=15pt
\newcommand{\nc}{\newcommand}
\newcommand{\delete}[1]{}
\nc{\mfootnote}[1]{\footnote{#1}} 
\nc{\todo}[1]{\tred{To do:} #1}

\nc{\mlabel}[1]{\label{#1}}  
\nc{\mcite}[1]{\cite{#1}}  
\nc{\mref}[1]{\ref{#1}}  
\nc{\meqref}[1]{\eqref{#1}} 
\nc{\mbibitem}[1]{\bibitem{#1}} 

\delete{
\nc{\mlabel}[1]{\label{#1}  
{\hfill \hspace{1cm}{\bf{{\ }\hfill(#1)}}}}
\nc{\mcite}[1]{\cite{#1}{{\bf{{\ }(#1)}}}}  
\nc{\mref}[1]{\ref{#1}{{\bf{{\ }(#1)}}}}  
\nc{\meqref}[1]{\eqref{#1}{{\bf{{\ }(#1)}}}} 
\nc{\mbibitem}[1]{\bibitem[\bf #1]{#1}} 
}

\newcommand {\comment}[1]{{\marginpar{*}\scriptsize\textbf{Comments:} #1}}
\nc{\mrm}[1]{{\rm #1}}
\nc{\id}{\mrm{id}}  \nc{\Id}{\mrm{Id}}

\def\a{\alpha}
\def\b{\beta}
\def\bd{\boxdot}
\def\bbf{\bar{f}}
\def\bF{\bar{F}}
\def\bbF{\bar{\bar{F}}}
\def\bbbf{\bar{\bar{f}}}
\def\bg{\bar{g}}
\def\bG{\bar{G}}
\def\bbG{\bar{\bar{G}}}
\def\bbg{\bar{\bar{g}}}
\def\bT{\bar{T}}
\def\bt{\bar{t}}
\def\bbT{\bar{\bar{T}}}
\def\bbt{\bar{\bar{t}}}
\def\bR{\bar{R}}
\def\br{\bar{r}}
\def\bbR{\bar{\bar{R}}}
\def\bbr{\bar{\bar{r}}}
\def\bu{\bar{u}}
\def\bU{\bar{U}}
\def\bbU{\bar{\bar{U}}}
\def\bbu{\bar{\bar{u}}}
\def\bw{\bar{w}}
\def\bW{\bar{W}}
\def\bbW{\bar{\bar{W}}}
\def\bbw{\bar{\bar{w}}}
\def\btl{\blacktriangleright}
\def\btr{\blacktriangleleft}
\def\ci{\circ}
\def\d{\delta}
\def\dd{\diamondsuit}
\def\D{\Delta}
\def\G{\Gamma}
\def\g{\gamma}
\def\k{\kappa}
\def\l{\lambda}
\def\lr{\longrightarrow}
\def\o{\otimes}
\def\om{\omega}
\def\p{\psi}
\def\r{\rho}
\def\ra{\rightarrow}
\def\rh{\rightharpoonup}
\def\lh{\leftharpoonup}
\def\s{\sigma}
\def\st{\star}
\def\ti{\times}
\def\tl{\triangleright}
\def\tr{\triangleleft}
\def\v{\varepsilon}
\def\vp{\varphi}
\def\vth{\vartheta}

\newtheorem{thm}{Theorem}[section]
\newtheorem{lem}[thm]{Lemma}
\newtheorem{cor}[thm]{Corollary}
\newtheorem{pro}[thm]{Proposition}
\theoremstyle{definition}
\newtheorem{defi}[thm]{Definition}
\newtheorem{ex}[thm]{Example}
\newtheorem{rmk}[thm]{Remark}
\newtheorem{pdef}[thm]{Proposition-Definition}
\newtheorem{condition}[thm]{Condition}
\newtheorem{question}[thm]{Question}
\renewcommand{\labelenumi}{{\rm(\alph{enumi})}}
\renewcommand{\theenumi}{\alph{enumi}}

\nc{\ts}[1]{\textcolor{purple}{Tianshui:#1}}

\font\cyr=wncyr10

 \title[Infinitesimal (BiHom-)bialgebras of any weight]{\bf Infinitesimal (BiHom-)bialgebras of any weight (I): Basic definitions and properties}

 \author[Ma]{Tianshui Ma}
 \address{School of Mathematics and Information Science, Henan Normal University, Xinxiang 453007, China}
         \email{matianshui@htu.edu.cn}

 \author[Makhlouf]{Abdenacer Makhlouf\textsuperscript{*}}
 \address{Universit{\'e} de Haute Alsace, IRIMAS-D\'epartement  de Math{\'e}matiques,18  rue des Fr{\`e}res Lumi{\`e}re F-68093 Mulhouse, France}

         \email{abdenacer.makhlouf@uha.fr}

  \thanks{\textsuperscript{*}Corresponding author}

\date{\today}

\begin{abstract} 
The purpose of this paper is to introduce and  study  $\lambda$-infinitesimal BiHom-bialgebras (abbr. $\l$-infBH-bialgebra) and some related structures. They can be seen as an extension of  $\l$-infinitesimal bialgebras considered by Ebrahimi-Fard, including Joni and Rota's  infinitesimal bialgebras  as well as  Loday and Ronco's infinitesimal bialgebras,  and including also infinitesimal BiHom-bialgebras introduced  by Liu, Makhlouf, Menini, Panaite. In this paper, we provide various relevant constructions and new concepts.   Two ways are provided for a unitary (resp. counitary) algebra (coalgebra) to be a $\l$-infBH-bialgebra and the  notion of $\l$-infBH-Hopf module is introduced and discussed. It is proved, in connexion with nonhomogeneous (co)associative BiHom-Yang-Baxter equation, that every (left BiHom-)module (resp. comodule) over a (anti-)quasitriangular (resp. (anti-)coquasitriangular) $\l$-infBH-bialgebra carries a structure of $\l$-infBH-Hopf module. Moreover, two approaches to construct BiHom-pre-Lie (co)algebras from $\l$-infBH-bialgebras are presented.
\end{abstract}

 \keywords{Infinitesimal bialgebra;infinitesimal Hopf module;quasitriangular infinitesimal bialgebra}

\subjclass[2020]{
17B61,
17D30,
17B38,
17A30,
16T10.}

 \maketitle

\tableofcontents

\numberwithin{equation}{section}
\allowdisplaybreaks

\section{Introduction and Preliminaries}
 A triple $(A, \mu, \D)$ is called an infinitesimal bialgebra if $(A, \mu)$ is an associative algebra, $(A, \D)$ is a coassociative coalgebra such that $\D$ is a derivation, i.e., for all $a, b \in A$, $\D(ab)=a\D(b)+\D(a)b$. This concept was introduced by Joni and Rota in connection with the calculus of divided differences \cite{JR}. In \cite{Ag99}, Aguiar developed the basic theory of infinitesimal Hopf (bialgebra) algebras. Infinitesimal versions of bicrossproduct, quasitriangular bialgebra, Yang-Baxter equation and Drinfeld double are defined. It is shown that the path algebra of an arbitrary quiver admits a canonical structure of infinitesimal Hopf algebra. Hom-version of infinitesimal bialgebra
was considered  in \cite{Yau10}, BiHom-version in \cite{LMMP3} and  braided version in \cite{WW14}. Other studies and results related to infinitesimal bialgebras can be found in  \cite{Ag00b,Ag04,GW,MY,MLiY,MYZZ}. The notion of infinitesimal bialgebra of weight $-1$ was introduced by Loday and Ronco \cite{LR06}  in order to obtain a structure theorem of cofree Hopf algebras, which is a part of a 2-associative bialgebra. An example can be constructed through  polynomial algebra $K[x]$ with the coproduct $\D(x^n)=\sum_{p=1}^{n}x^{p}\o x^{n-p}$.

 A generalization including both cases was provided by Ebrahimi-Fard in \cite{EF06}, where infinitesimal bialgebra of weight $\l$ (abbr. $\l$-inf bialgebra) are defined. More precisely, a $\l$-inf bialgebra is a triple $(A, \mu, \D)$ consisting of an algebra $(A, \mu)$ (possibly without a unit) and a coalgebra $(A, \D)$ (possibly without a counit) such that for all $a, b \in A, \D(ab)=a\D(b)+\D(a)b+\l a\o b$. A solution of the nonhomogeneous associative Yang-Baxter equation \cite{OP10,EF06} can produce a $\l$-inf bialgebra. Further results about $\l$-inf bialgebras can be found in\cite{EF06,Fo09,Fo10,LR06,ZCGL18,ZG,ZGZL}. It is worth mentioning that a unitary and counitary $0$-inf bialgebra $H$ is trivial, i.e., $H=0$ and $\v(1)=0$ \cite{Ag99}, but for non-zero weight cases, this result does not hold. For example, if the weight $\l$ of infinitesimal bialgebra $H$ is $-1$ and suppose that $H$ is unitary and counitary, then the counit $\v$ is an algebra map. Examples of (co)unitary $\l$-inf bialgebras can be constructed on  the space of decorated planar rooted forests \cite{ZGL,ZCGL18}.

 Zhang and Guo introduced  in \cite{ZG} the concept of $\l$-infinitesimal Hopf modules and proved that any module over a unitary quasitriangular $\l$-inf bialgebra owns a structure of $\l$--infinitesimal Hopf module. They also derived two pre-Lie algebras from a $\l$-inf bialgebra and commutative $\l$-inf bialgebra. Also in \cite{ZG}, an example shows that any unitary algebra $(A, \mu, 1)$ is a unitary $\l$-inf bialgebra by taking $\D(a)=-\l (a\o 1)$. A very interesting phenomenon that we find is that if we define a coproduct on a unitary algebra $(A, \mu, 1)$ by $\overline{\D}(a)=-\l(1\o a)$ (just exchanging the positions of a and $1$), then $(A, \mu, 1, \overline{\D})$ is also a  unitary $\l$-inf bialgebra. Enlightened by this phenomenon, we introduce the notion of anti-quasitriangular $\l$-inf bialgebra (see Definition \ref{de:14.6a}) which is different from the quasitriangular case based on their equivalent characterizations (see Propositions \ref{pro:14.8} and \ref{pro:14.25}). When $\l=-1$, the equivalent characterizations for an anti-quasitriangular $\l$-inf bialgebra in Proposition  \ref{pro:14.25} are consistent with \cite[Lemma 3.19]{Br}.  Furthermore  we show  that any module over a unitary anti-quasitriangular $\l$-inf bialgebra still has a structure of $\l$-infinitesimal Hopf module and its comodule action depends on the weight $\l$ (see Theorem \ref{thm:12.02a}). Anti-quasitriangular $0$-inf bialgebra and quasitriangular $0$-inf bialgebra are consistent. One purpose of this paper is to explore this new phenomenon.

 In \cite{GMMP}, Graziani, Makhlouf, Menini and Panaite introduced algebras in a group Hom-category, which are called BiHom-associative algebras and involve two commuting multiplicative linear maps. BiHom-type algebras can be seen as an extension of Hom-type algebras, which first arose in quasi-deformations of Lie algebras of vector fields and lead to the concept of Hom-Lie algebras introduced by Hartwig, Larsson and Silvestrov in \cite{HLS}, see also \cite{MS}. The examples of BiHom-type of algebras can be provided by the ``Yau twist principle". In 2020, Liu, Makhlouf, Menini and Panaite proposed and studied (quasitriangular) Joni-Rota's infinitesimal BiHom-bialgebras \cite{LMMP3} generalizing the Hom-case studied by Yau  in \cite{Yau10}. In \cite{MY}, Ma and Yang presented the Drinfeld double for infinitesimal BiHom-bialgebra and Ma, Li and Yang studied the properties of infinitesimal BiHom-bialgebra from BiHom-coderivations \cite{MLiY}. However, the weight of infinitesimal BiHom-bialgebras studied above is zero. The other purpose of this  paper is to extend these studies and consider the case of infinitesimal BiHom-bialgebra of any weight $\l$.

 The layout of the paper is as follows. In Section \ref{se:infbhbialg}, we extend Ebrahimi-Fard's $\l$-inf bialgebras   \cite{EF06} to the BiHom-case, which includes Joni and Rota's infinitesimal bialgebras \cite{JR}, Loday and Ronco's infinitesimal bialgebra \cite{LR06} and Liu, Makhlouf, Menini, Panaite's infinitesimal BiHom-bialgebras \cite{LMMP3}. We prove that every unitary BiHom-associative algebra $(A, \mu, 1, \alpha, \beta)$ (resp. counitary BiHom-coassociative coalgebra $(A, \Delta, \psi, \omega))$ possesses $\lambda$-infBH-bialgebra structures with the comultiplication $\D: A\lr A\o A$ by $\Delta(a)=-\lambda(\omega(a)\otimes 1)$ or $\Delta(a)=-\lambda(1\otimes \psi(a))\big)$ (Example \ref{ex:12.3}) (resp. multiplication $\mu: A\o A\lr A$ by $\mu(a\otimes b)=-\lambda\alpha(a)\varepsilon(b)$ or $\mu(a\otimes b)=-\lambda\varepsilon(a)\beta(b)$ (Example \ref{ex:12.35})). These two examples arouse the new constructions of $\lambda$-infBH-bialgebra by an element $r\in A\o A$ (Definitions \ref{de:14.6} and \ref{de:14.6a}) (resp. $\s\in (A\o A)^*$ (Definition \ref{de:13.004})). Let $(A,\mu_{A}, \Delta_{A}, \varepsilon_{A}, \alpha_{A}, \beta_{A}, \psi_{A}, \omega_{A})$ (resp. $(C,\mu_{C},\Delta_{C},1_{C},\alpha_{C},\beta_{C},\psi_{C},\omega_{C})$) be a counitary (resp. unitary) $\lambda$-infBH-bialgebra. We provide a new tensor product algebra (resp. coalgebra) structure of two algebras (resp. coalgebras) such that $\Delta_{A}: (A,\mu_{A},\alpha_{A},\beta_{A})\longrightarrow (A\otimes A,\cdot_{\varepsilon},\alpha_{A}\otimes\alpha_{A},\beta_{A}\otimes\beta_{A})$ (resp. $\mu_{C}: (C\otimes C,\Delta_{\eta},\psi_{C}\otimes\psi_{C},\omega_{C}\otimes\omega_{C})\longrightarrow (C,\Delta_{C},\psi_{C},\omega_{C})$) is a morphism of BiHom-associative algebras (resp. coalgebras) (see Theorem \ref{thm:12.12}).

 In Section \ref{se:infbhmod} we investigate  $\l$-infBH-Hopf modules. First, we provide some relevant and  interesting examples. Then considering  a unitary BiHom-algebra $(A, \mu, 1, \a, \b)$, an element $r\in A\o A$ and $\psi, \om: A\lr A$ two linear maps, we define a linear map $\D_r: A\lr A\o A$ by
  $
 \Delta_{r}(a)=\alpha^{-1}(a)\triangleright r-r\triangleleft \beta^{-1}(a)-\lambda (\omega(a)\otimes 1),~~~\forall~~a\in A
$
 (Eq.(\ref{eq:14.5})) and
  prove that $\Delta_{r}$ is a $\lambda$-BiHom-derivation and also find the equivalent condition for $\D_r$ is BiHom-coassociative (Theorem \ref{thm:14.5}). Then we introduce the notions of $\l$-associative BiHom-Yang-Baxter equation (abbr. $\l$-abhYBe) and quasitriangular unitary $\l$-infBH-bialgebra. Therefore we investigate  equivalent characterizations of quasitriangular unitary $\l$-infBH-bialgebras (see Proposition \ref{pro:14.8}) and obtain that  modules over quasitriangular unitary $\lambda$-infBH-bialgebras carry  structures of $\lambda$-infBH-Hopf modules (Theorem \ref{thm:12.02}). Moreover, inspired by Example \mref{ex:12.3}, we  define a new comultiplication for $\l$-infBH-bialgebra by replacing $\l(\om(a)\o 1)$ in $\Delta_{r}$ above with $\lambda (1\otimes \psi(a))$, then a new interesting phenomenon appears. We notice  that $\widetilde{\Delta}_{r}$ defined by Eq.(\mref{eq:14.25}) is still a $\lambda$-BiHom-derivation (Proposition \ref{pro:14.21}). It turns out that the weight of the induced associative BiHom-Yang-Baxter equation is $-\l$ by the coassociativity of $\widetilde{\Delta}_{r}$ (Theorem \ref{thm:14.22}). We call this unitary $\l$-infBH-bialgebra determined by $\l$-abhYBe ``anti-quasitriangular" (Definition \ref{de:14.6a}). Comparing the equivalent characterizations of quasitriangular case in Proposition \ref{pro:14.8} and anti-quasitriangular case in Proposition \ref{pro:14.25}, we find the two cases are essentially different for the non-zero weight. While for zero weight, they are same. Furthermore, $\lambda$-infBH-Hopf modules also can be derived from the modules over anti-quasitriangular unitary $\lambda$-infBH-bialgebras. We notice here that the comodule coaction depends on the weight $\l$ (Theorem \ref{thm:12.02a}). Also in this section, we study  $\lambda$-infBH-bialgebras and $\l$-infBH-Hopf modules from the perspective of (BiHom-)coderivation, which covers and at the same time extends the results in \cite{MMS,MLiY}. More precisely, corresponding to the cases of unitary ``quasitriangular" and ``anti-quasitriangular" above, for a counitary BiHom-coalgebra $(C, \D, \v, \psi, \om)$ and an element $\sigma \in (C\otimes C)^{\ast}$, we define the multiplications $\mu_{\sigma}$ (Eq.\ref{eq:01.04}) (resp. $\widetilde{\mu_{\sigma}}$ (Eq.\ref{eq:01.08})), then we get the concept of $\pm \lambda$-coassociative BiHom-Yang-Baxter equation (abbr. $\pm \lambda$-coabhYBe) (Definition \ref{de:12.03}). We show that  $\lambda$-infBH-Hopf modules can be derived from every  comodule over (anti-)coquasitriangular counitary $\lambda$-infBH-bialgebra (Theorem \ref{thm:12.013}). Lastly, we propose two ways to construct BiHom-pre-Lie algebras from $\l$-infBH-bialgebras. The structure in Theorem \ref{thm:13.3} makes the following diagram commutative.
$$\hspace{10mm} \xymatrix@C4.3em{
&\text{Quasitriangular}\atop \text{$\l$-infBH Bialgebras}
\ar@2{->}^{\text{Lemma}\ \mref{lem:014.10}}[r]
\ar@2{->}_{\text{Corollary}\  \mref{cor:waybe}}[d]
&\text{$\lambda$-Rota Baxter}\atop \text{BiHom-algebras}
\ar@2{->}^{\text{Lemma}\ \mref{lem:14.11}}[r]
&\text{BiHom-Dendriform}\atop \text{algebras}\\
&\text{$\lambda$-infBH}\atop \text{Bialgebras}
\ar@2{->}^{\text{Theorem}\  \mref{thm:13.3}}[rr]
&&\text{BiHom-Pre-Lie}\atop \text{algebras}
\ar@2{<-}^{\text{\cite[Proposition 3.6]{LMMP6}}}[u]
&&}
$$
 We also put forward two methods to construct BiHom-pre-Lie coalgebras from $\l$-infBH-bialgebras, one of which can also provide a commutative diagram corresponding to the above one.

 Throughout this paper, $K$ will be a field, and all vector spaces, tensor products, and homomorphisms are over $K$. We use Sweedler's notation for terminology on coalgebras. For a coalgebra $C$ with a comultiplication $\Delta$, we write  $\Delta(c)=c_1\otimes c_2$, for all $c\in C$.
 We denote by $\id_M$ the identity map from $M$ to $M$ and by  $\tau: M\o N\ra N\o M$ the flip map. We abbreviate "infinitesimal BiHom-" to "infBH-".

 Let us recall from \cite{GMMP,MLi} the following basic definitions and structures.
 \begin{defi}\mlabel{de:1.1} A {\bf BiHom-associative algebra} is a 4-tuple $(A,\mu,\alpha,\beta)$, where $A$ is a linear space, $\alpha,\beta:A\lr A$ and $\mu:A\otimes A\lr A$ (write $\mu(a\otimes b)=ab$) are linear maps satisfying the following conditions, for all $a,b,c\in A$ :
 \begin{eqnarray}
 &\alpha\circ\beta=\beta\circ\alpha,\quad \alpha(ab)=\alpha(a)\alpha(b),\quad\beta(ab)=\beta(a)\beta(b),&\mlabel{eq:1.2}\\
 &\alpha(a)(bc)=(ab)\beta(c).&\mlabel{eq:1.3}
 \end{eqnarray}
 \end{defi}
 A BiHom-associative algebra $(A,\mu,\alpha,\beta)$ is called {\bf unitary} if there exists an element $1_{A}\in A$ (called a unit) such that
 \begin{eqnarray}
 &\alpha(1_{A})=1_{A},\quad\beta(1_{A})=1_{A},\quad a1_{A}=\alpha(a),\quad 1_{A}a=\beta(a),\quad \forall a\in A.&\mlabel{eq:1.5}
 \end{eqnarray}

 A morphism $f:(A,\mu_{A},\alpha_{A},\beta_{A})\longrightarrow (B,\mu_{B},\alpha_{B},\beta_{B})$ of BiHom-associative algebras is a linear map $f:A\longrightarrow B$ such that $\alpha_{B}\circ f=f\circ\alpha_{A}$, $\beta_{B}\circ f=f\circ\beta_{A}$ and $f\circ\mu_{A}=\mu_{B}\circ(f\otimes f)$.

 \begin{rmk}
 {\bf ``Yau twist"}: Let $(A, \mu)$ be an associative algebra, $\a, \b: A\lr A$ two linear maps satisfying Eqs.(\mref{eq:1.2}). Then $(A, \mu\ci (\a\o \b), \a, \b)$ is a BiHom-associative algebra.
 \end{rmk}

 \begin{defi}\mlabel{de:1.2} A {\bf BiHom-coassociative coalgebra} is a 4-tuple $(C,\Delta,\psi,\omega)$, in which $C$ is a linear space, $\psi,\omega:C\lr C$ and $\Delta:C\lr C\otimes C$ are linear maps, such that
 \begin{eqnarray}
 &\psi\circ\omega=\omega\circ \psi,~~~(\psi\otimes\psi)\circ\Delta=\Delta\circ\psi,~~
 (\omega\otimes\omega)\circ\Delta=\Delta\circ\omega,&\mlabel{eq:1.7}\\
 &(\Delta\otimes\psi)\circ\Delta=(\omega\otimes \Delta)\circ\Delta.&\mlabel{eq:1.9}
 \end{eqnarray}
 \end{defi}

 A BiHom-coassociative coalgebra $(C,\Delta,\psi,\omega)$ is called {\bf counitary} if there exists a linear map $\varepsilon: C\lr K$ (called a counit) such that
 \begin{eqnarray}
 &\varepsilon\circ\psi=\varepsilon,\quad \varepsilon\circ\omega=\varepsilon,\quad (\id_{C}\otimes\varepsilon)\circ\Delta=\omega,\quad(\varepsilon\otimes \id_{C})\circ\Delta=\psi.&\mlabel{eq:1.11}
 \end{eqnarray}

 A morphism $g:(C,\Delta_{C},\psi_{C},\omega_{C})\longrightarrow (D,\Delta_{D},\psi_{D},\omega_{D})$ of BiHom-coassociative coalgebras is a linear map $g:C\longrightarrow D$ such that $\omega_{D}\circ g=g\circ\omega_{C}$, $\psi_{D}\circ g=g\circ\psi_{C}$ and $\Delta_{D}\circ g= (g\otimes g)\circ\Delta_{C}$.

 \begin{defi}\mlabel{de:1.3} Let $(A,\mu,\alpha_{A},\b_{A})$ be a BiHom-associative algebra. A {\bf left $(A,\mu,\alpha_{A},\beta_{A})$-module} is a 4-tuple $(M,\g,\alpha_{M},\beta_{M})$,  where $M$ is a linear space,  $\alpha_{M}, \beta_{M}: M\lr M$ and $\g: A\otimes M\lr M$ (write $\g(a\o m)=a\tl m$) are linear maps such that, for all $a, a'\in A, m\in M$,
 \begin{eqnarray}
 &\alpha_{M}\circ\beta_{M}=\beta_{M}\circ\alpha_{M},~~\alpha_{M}(a\triangleright m)=\alpha_{A}(a)\triangleright\alpha_{M}(m),~~\beta_{M}(a\triangleright m)=\beta_{A}(a)\triangleright\beta_{M}(m),&\mlabel{eq:1.13}\\
 &\alpha_{A}(a)\triangleright(a'\triangleright m)=(aa')\triangleright\beta_{M}(m).&\mlabel{eq:1.15}
 \end{eqnarray}
 Likewise, we can get the right version of $(A,\mu,\alpha_{A},\beta_{A})$-module.

 If $(M,\g,\alpha_{M},\beta_{M})$ is a left $(A,\mu,\alpha_{A},\beta_{A})$-module and at the same time $(M,\nu,\alpha_{M},\beta_{M})$ is a right $(A,\mu,\alpha_{A},\beta_{A})$-module (write $\nu(m\o a)=m\tr a$), then $(M,\g,\nu,\alpha_{M},\beta_{M})$ is an {\bf $(A,\mu,\alpha_{A},\beta_{A})$-bimodule} if
 \begin{eqnarray}
 &\alpha_{A}(a)\triangleright(m\triangleleft a')=(a\triangleright m)\triangleleft \beta_{A}(a').&\mlabel{eq:1.16}
 \end{eqnarray}
 \end{defi}

 \begin{pro}\mlabel{pro:14.1} Let $(A, \mu, \alpha_{A}, \beta_{A})$ be a BiHom-associative algebra and $(M, \g_M,\nu_M, \alpha_{M}, \beta_{M})$, $(N, \g_N,\nu_N$, $\alpha_{N}, \beta_{N})$, $(V, \g_V,\nu_V, \alpha_{V}, \beta_{V})$ be $(A, \mu, \alpha_{A}, \beta_{A})$-bimodules. $\psi_{A}, \omega_{A}: A\lr A$ are linear maps such that $\psi_{A}(aa')=\psi_{A}(a)\psi_{A}(a')$, $\omega_{A}(aa')=\omega_{A}(a)\omega_{A}(a')$ and any two of the maps $\alpha_{A}$, $\beta_{A}$, $\psi_{A}$, $\omega_{A}$ commute. We consider the following left and right actions of $A$ on $M\otimes N\otimes V$, for all $a\in A, m\in M, n\in N, v\in V$
 \begin{eqnarray*}
 &a\triangleright(m\otimes n\otimes v)=\omega_{A}(a)\triangleright m\otimes\beta_{N}(n)\otimes\beta_{V}(v),&\\
 &(m\otimes n\otimes v)\triangleleft a=\alpha_{M}(m)\otimes \alpha_{N}(n)\otimes v\triangleleft\psi_{A}(a).&
 \end{eqnarray*}
 Then $(M\otimes N\otimes V, \triangleright, \triangleleft, \alpha_{M}\otimes\alpha_{N}\otimes\alpha_{V}, \beta_{M}\otimes\beta_{N}\otimes\beta_{V})$ is an $(A,\mu,\alpha_{A},\beta_{A})$-bimodule.
 \end{pro}

 \begin{rmk}\mlabel{rmk:14.2} (1) As  special cases of Proposition \mref{pro:14.1}, for all $a, x, y, z\in A$, we have the following actions of $A$ on $A\otimes A$ and $A\otimes A\otimes A$ given by
 \begin{eqnarray}
 &a\triangleright(x\otimes y)=\omega(a)x\otimes\beta(y),&\mlabel{eq:14.1}\\
 &(x\otimes y)\triangleleft a=\alpha(x)\otimes y\psi(a),&\mlabel{eq:14.2}
 \end{eqnarray}
 and
 \begin{eqnarray}
 &a\triangleright(x\otimes y\otimes z)=\omega(a)x\otimes\beta(y)\otimes\beta(z),&\mlabel{eq:14.3}\\
 &(x\otimes y\otimes z)\triangleleft a=\alpha(x)\otimes\alpha(y)\otimes z\psi(a).&\mlabel{eq:14.4}
 \end{eqnarray}
 respectively.

 (2) The actions defined in Eqs.(\mref{eq:14.1}) and (\mref{eq:14.2}) coincide with the actions in \cite[Lemma 4.2]{LMMP3}.
 \end{rmk}

\section{$\lambda$-infinitesimal BiHom-bialgebras}\label{se:infbhbialg}
 In this section, we introduce the notion of $\lambda$-infinitesimal BiHom-bialgebra which is the BiHom-version of Ebrahimi-Fard's $\l$-infinitesimal bialgebra \cite{EF06} including Joni and Rota's infinitesimal bialgebra \cite{JR}, Loday and Ronco's infinitesimal bialgebra \cite{LR06} and Liu, Makhlouf, Menini, Panaite's infBH-bialgebra \cite{LMMP3} as special cases.

\subsection{Definitions}
 \begin{defi}\mlabel{de:12.1} Let $\lambda$ be a given element of $K$. (1) A {\bf $\lambda$-infinitesimal BiHom-bialgebra} (abbr. $\lambda$-infBH-bialgebra) is a 7-tuple $(A, \mu, \Delta, \alpha, \beta, \psi, \omega)$ such that $(A, \mu, \alpha, \beta)$ is a BiHom-associative algebra and  $(A, \Delta, \psi, \omega)$ is a BiHom-coassociative coalgebra satisfying the following conditions
 \begin{eqnarray}
 &\alpha\circ\psi=\psi\circ\alpha,\ \alpha\circ\omega=\omega\circ\alpha,\ \beta\circ\psi=\psi\circ\beta,\ \beta\circ\omega=\omega\circ\beta,&\mlabel{eq:12.1}\\
 &(\alpha\otimes\alpha)\circ\Delta=\Delta\circ\alpha,\ (\beta\otimes\beta)\circ\Delta=\Delta\circ\beta,&\mlabel{eq:12.2}\\
 &\psi\circ\mu=\mu\circ(\psi\otimes\psi),\ \omega\circ\mu=\mu\circ(\omega\otimes\omega),&\mlabel{eq:12.3}\\
 &\Delta\circ\mu=(\mu\otimes\beta)\circ(\omega\otimes\Delta)
 +(\alpha\otimes\mu)\circ(\Delta\otimes\psi)+\lambda(\alpha\omega\otimes\beta\psi).&\mlabel{eq:12.4}
 \end{eqnarray}

 (2) If further $(A, \mu, 1, \alpha, \beta)$ is a unitary BiHom-associative algebra, then the 8-tuple $(A, \mu, 1, \Delta,$ $\alpha, \beta, \psi, \omega)$ is called a {\bf unitary $\lambda$-infBH-bialgebra} if
 \begin{eqnarray}
 & \psi(1_{A})=1_{A},\ \omega(1_{A})=1_{A}.&\mlabel{eq:12.30}
 \end{eqnarray}

 (3) If further $(A, \Delta, \varepsilon, \psi, \omega)$ is a counitary BiHom-coassociative coalgebra, then the 8-tuple $(A, \mu, \Delta, \varepsilon, \alpha, \beta, \psi, \omega)$ is called a {\bf counitary $\lambda$-infBH-bialgebra} if
 \begin{eqnarray}
 &\varepsilon\circ\alpha=\varepsilon,\ \varepsilon\circ\beta=\varepsilon.&\mlabel{eq:12.31}
 \end{eqnarray}
 \end{defi}

 \begin{rmk}\mlabel{rmk:12.2} (1) If $\lambda=0$ in Definition \mref{de:12.1}, then we get infBH-bialgebras introduced by Liu, Makhlouf, Menini, Panaite in \cite[Definition 4.1]{LMMP3} and also studied in \cite{MLi,MLiY,MY}. If further $\alpha=\beta=\psi=\omega=\id$, then one can obtain Joni and Rota's infinitesimal bialgebras \cite{JR}.

 (2) If $\l=-1$ in Definition \mref{de:12.1}, then we have the BiHom-version of Loday and Ronco's infinitesimal bialgebra \cite{LR06}.

 (3) Definition \mref{de:12.1} is the BiHom-version of Ebrahimi-Fard's $\l$-infinitesimal bialgebras \cite{EF06} studied in \cite{ZG}.

 (4) A morphism between $\l$-infBH-bialgebras is a linear map that commutes with the structure maps $\a, \b, \psi, \omega$, the multiplication $\mu$, and the comultiplication $\D$.
 \end{rmk}

 The following theorem can provide  examples of $\l$-infBH-bialgebras from $\lambda$-inf bialgebras.

 \begin{thm}\mlabel{thm:12.48} Let $(A,\mu,\Delta)$ be a $\lambda$-infinitesimal bialgebra \cite{EF06} and $\alpha,\beta,\psi,\omega: A\longrightarrow A$ be morphisms of algebras and coalgebras such that any two of them commute. Then $A_{(\alpha,\beta,\psi,\omega)}:=(A,\mu_{(\alpha,\beta)}:=\mu\circ(\alpha\otimes\beta),
 \Delta_{(\psi,\omega)}:=(\omega\otimes\psi)\circ\Delta,\alpha,\beta,\psi,\omega)$ is a $\lambda$-infBH-bialgebra, called the Yau twist of $(A,\mu,\Delta)$.
 \end{thm}

 \begin{proof} The fact that $(A,\mu_{(\alpha,\beta)},\alpha,\beta)$ is a BiHom-associative algebra and $(A,\Delta_{(\psi,\omega)},\psi,\omega)$ is a BiHom-coassociative coalgebra is known from \cite{GMMP}. By the assumption, we have Eqs.(\mref{eq:12.1})-(\mref{eq:12.3}). Next we only need to prove Eq.(\mref{eq:12.4}). For simplicity, we denote $\mu_{(\alpha,\beta)}(a\otimes b)=a\ast b=\alpha(a)\cdot\beta(b)$ and $\Delta_{(\psi,\omega)}(a)=a_{[1]}\otimes a_{[2]}=\omega(a_{1})\otimes\psi(a_{2})$, for all $a,b\in A$, we compute:
 \begin{eqnarray*}
 &&\hspace{-13mm}\omega(a)\ast b_{[1]}\otimes\beta(b_{[2]})+\alpha(a_{[1]})\otimes a_{[2]}\ast\psi(b)+ \lambda\alpha\omega(a)\otimes \beta\psi(b)\\
 &=&\alpha\omega(a)\beta\omega(b_{1})\otimes\beta\psi(b_{2})+\alpha\omega(a_{1})\otimes \alpha\psi(a_{2})\beta\psi(b)+ \lambda\alpha\omega(a)\otimes \beta\psi(b)\\
 &=&\omega(\alpha(a)\beta(b_{1}))\otimes\beta\psi(b_{2})+\alpha\omega(a_{1})\otimes \psi(\alpha(a_{2})\beta(b))+ \lambda\alpha\omega(a)\otimes \beta\psi(b)\\
 &=&(\omega\otimes\psi)(\alpha(a)\beta(b_{1}) \otimes\beta(b_{2})+\alpha(a_{1})\otimes \alpha(a_{2})\beta(b)+ \lambda\alpha(a)\otimes \beta(b))\\
 &\stackrel{}=&(\omega\otimes\psi)\Delta(\alpha(a) \beta(b))\\
 &=&\Delta_{(\psi,\omega)}(a\ast b),
 \end{eqnarray*}
 finishing the proof.
 \end{proof}

 The following two examples show that every (co)unitary BiHom-(co)associative (co)algebra possesses $\lambda$-infBH-bialgebra structures.

 \begin{ex}\mlabel{ex:12.3} Let $(A, \mu, 1, \alpha, \beta)$ be a unitary BiHom-associative algebra, $\psi,\omega: A\longrightarrow A$ be linear maps such that Eqs.(\mref{eq:12.1}), (\mref{eq:12.3}) and (\mref{eq:12.30}) hold. Define the comultiplication $\D: A\lr A\o A$ by
 \begin{eqnarray*}
 &\Delta(a)=-\lambda(\omega(a)\otimes 1)\ \ \big(\hbox{resp.}\ \Delta(a)=-\lambda(1\otimes \psi(a))\big)&
 \end{eqnarray*}
 for all $a\in A$. Then $(A, \mu, 1, \D, \a, \b, \psi,\omega)$ is a unitary $\lambda$-infBH-bialgebra.
 \end{ex}

 \begin{proof} We only check the case of $\Delta(a)=-\lambda(\omega(a)\otimes 1)$, another case is similar. For all $a, b \in A$, we calculate
 \begin{eqnarray*}
 (\Delta\otimes\psi)\circ\Delta(a)&=&-\lambda(\Delta\otimes\psi)(\omega(a)\otimes 1)=\lambda^{2}(\omega^{2}(a)\otimes 1\otimes 1)\\
 &\stackrel{(\mref{eq:12.30})}=&\lambda^{2}(\omega^{2}(a)\otimes \omega(1)\otimes 1)=-\lambda(\omega\otimes\Delta)(\omega(a)\otimes 1)\\
 &=&(\omega\otimes\Delta)\circ\Delta(a).
 \end{eqnarray*}
 Therefore, BiHom-coassociativity of $(A, \Delta, \psi, \omega)$ holds. Next, we have
 \begin{eqnarray*}
 &&\omega(a)b_{1}\otimes\beta(b_{2})+\alpha(a_{1})\otimes a_{2}\psi(b)+ \lambda(\alpha\omega(a)\otimes\beta\psi(b))\\
 &&\qquad\stackrel{(\mref{eq:1.5})}=-\lambda\omega(a)\omega(b)\otimes 1-\lambda\alpha\omega(a)\otimes \beta\psi(b)+ \lambda(\alpha\omega(a)\otimes\beta\psi(b))\\
 &&\qquad\stackrel{(\mref{eq:12.3})}=-\lambda\omega(ab)\otimes 1=\Delta(ab),
 \end{eqnarray*}
 finishing the proof.
 \end{proof}

 \begin{ex}\mlabel{ex:12.35} Let $(A, \Delta, \varepsilon, \psi, \omega)$ be a counitary BiHom-coassociative coalgebra and $\alpha,\beta: A\longrightarrow A$ be linear maps satisfying Eqs.(\mref{eq:12.1}), (\mref{eq:12.2}) and (\mref{eq:12.31}). If we define the multiplication $\mu: A\o A\lr A$ by
 \begin{eqnarray*}
 &\mu(a\otimes b)=-\lambda\alpha(a)\varepsilon(b)\ \ \big(\hbox{resp.}\ \mu(a\otimes b)=-\lambda\varepsilon(a)\beta(b) \big)&
 \end{eqnarray*}
 for all $a,b\in A$, then $(A, \mu, \D, \v, \a, \b, \psi,\omega)$ is a counitary $\lambda$-infBH-bialgebra.
 \end{ex}

 \begin{proof} Here we prove the case of $\mu(a\otimes b)=-\lambda\varepsilon(a)\beta(b)$, the other case is left to the reader. For all $a, b \in A$, we calculate
 \begin{eqnarray*}
 \alpha(a)(bc)&=&-\lambda\alpha(a)(\varepsilon(b)\beta(c))
 \stackrel{(\mref{eq:12.31})}=\lambda^{2}\varepsilon(a)\varepsilon(b)\beta^{2}(c)\\
 &\stackrel{(\mref{eq:12.31})}=&-\lambda(\varepsilon(a)\beta(b))\beta(c)=(ab)\beta(c).
 \end{eqnarray*}
 So BiHom-associativity of $(A, \mu, \alpha, \beta)$ holds. The compatibility condition can be verified as follows.
 \begin{eqnarray*}
 &&\omega(a)b_{1}\otimes\beta(b_{2})+\alpha(a_{1})\otimes a_{2}\psi(b)+ \lambda(\alpha\omega(a)\otimes\beta\psi(b))\\
 &&\qquad=-\lambda\varepsilon\omega(a)\beta(b_{1})\otimes\beta(b_{2})-\lambda\alpha(a_{1})\otimes \varepsilon(a_{2})\beta\psi(b)+ \lambda(\alpha\omega(a)\otimes\beta\psi(b))\\
 &&\qquad\stackrel{(\mref{eq:1.11})}=-\lambda\varepsilon(a)\beta(b_{1})\otimes\beta(b_{2})-\lambda \alpha\omega(a)\otimes\beta\psi(b)+ \lambda(\alpha\omega(a)\otimes\beta\psi(b))\\
 &&\qquad\stackrel{(\mref{eq:12.3})}=-\lambda\varepsilon(a)\beta(b_{1})\otimes\beta(b_{2})=\Delta(ab),
 \end{eqnarray*}
 as desired.
 \end{proof}

 \begin{rmk}
 We remark here that there are two cases containing in Example \mref{ex:12.3} (resp. \mref{ex:12.35}), which motivates the construction of anti-quasitriangular (resp. anti-coquasitriangular) infBH-bialgebras in Section \mref{se:module} (resp. \mref{se:comodule}).
 \end{rmk}

 In order to illustrate the self-duality of $\l$-infBH-bialgebras, we introduce the following concepts.

 \begin{defi}\mlabel{de:12.13} Let $\lambda$ be a given element of $K$, $(A,\mu,\alpha,\beta)$ be a BiHom-associative algebra, $\psi,\omega: A\longrightarrow A$ be two linear maps such that Eq.(\mref{eq:12.1}) holds. Then $\delta: A\longrightarrow A\otimes A$ (write $\delta(a)=a_{[1]}\otimes a_{[2]}$) is a {\bf $\lambda$-BiHom-derivation} if it satisfies
 \begin{eqnarray}
 &(\alpha\otimes\alpha)\circ\delta=\delta\circ\alpha,
  \ (\beta\otimes\beta)\circ\delta=\delta\circ\beta,
 \ (\psi\otimes\psi)\circ\delta=\delta\circ\psi,
  \ (\omega\otimes\omega)\circ\delta=\delta\circ\omega,&\mlabel{eq:12.9}\\
 &\delta(ab)=a\triangleright\delta(b)+\delta(a)\triangleleft b+\lambda\alpha\omega(a)\otimes\beta\psi(b).&\mlabel{eq:12.10}
 \end{eqnarray}

 By using the module action in Proposition \mref{pro:14.1}, Eq.(\mref{eq:12.10}) is exactly
 $$
 \delta(ab)=\omega(a)b_{[1]}\otimes\beta(b_{[2]})+\alpha(a_{[1]})\otimes a_{[2]}\psi(b)+\lambda\alpha\omega(a)\otimes\beta\psi(b),~\forall a,b\in A.
 $$
 \end{defi}

 Dually, we have

 \begin{defi}\mlabel{de:12.14} Let $\lambda$ be a given element of $K$, $(C, \D, \psi, \omega)$ be a BiHom-coassociative coalgebra, $\alpha,\beta: C\longrightarrow C$ be two linear maps such that Eq.(\mref{eq:12.1}) holds. Then $\vartheta: C\otimes C\longrightarrow C$ (write $\vth(c\otimes d)=c\diamond d $) is a {\bf $\lambda$-BiHom-coderivation} if it satisfies
 \begin{eqnarray}
 &\vth \circ(\omega\otimes \omega)=\omega \circ\vth,\ \vth \circ(\psi \otimes \psi)=\psi \circ\vth,\ \vth \circ(\alpha\otimes \alpha)=\alpha \circ\vth,\  \vth \circ(\beta \otimes \beta)=\beta \circ\vth,& \mlabel{eq:12.11}\\
 &(c\diamond d)_{1}\otimes (c\diamond d)_{2}=(\omega (c)\diamond d_{1})\otimes \beta (d_{2})+\alpha (c_{1})\otimes (c_{2}\diamond\psi (d))+\lambda\alpha\omega(c)\otimes\beta\psi(d),& \mlabel{eq:12.12}
 \end{eqnarray}
 for all $c, d\in C$.
 \end{defi}

 \begin{rmk} $0$-BiHom-(co)derivation was introduced in \cite{MY} (\cite{MLiY}). If further, $\a=\b=\psi=\omega=\id$, we can cover the classical case.
 \end{rmk}

 \begin{pro}\mlabel{pro:12.15} Let $(A,\mu,\alpha,\beta)$ be a BiHom-associative algebra and $(A,\Delta,\psi,\omega)$ a BiHom- coassociative coalgebra such that Eqs.(\mref{eq:12.1}), (\mref{eq:12.2}) and (\mref{eq:12.3}) hold. Then the following statements are equivalent:

 (1) $\Delta: A\longrightarrow A\otimes A$ is a $\lambda$-BiHom-derivation.

 (2) $\mu: A\otimes A\longrightarrow A$ is a $\lambda$-BiHom-coderivation.

 (3) Eq.(\mref{eq:12.4}) holds.
 \end{pro}

 \begin{proof} Note that the compatibility condition Eq.(\mref{eq:12.4}) can be written as
 \begin{eqnarray*}
 &\Delta(ab)=\omega(a)b_{1}\otimes\beta(b_{2})+\alpha(a_{1})\otimes a_{2}\psi(b)+ \lambda(\alpha\omega(a)\otimes\beta\psi(b)),&
 \end{eqnarray*}
 which implies that $\Delta: A\longrightarrow A\otimes A$ is a $\l$-BiHom-derivation in the BiHom-associative algebra $(A,\mu,\alpha,\beta)$, in other words, $\mu: A\otimes A\longrightarrow A$ is a $\l$-BiHom-coderivation in the BiHom-coassociative coalgebra $(A,\Delta,\psi,\omega)$.
 \end{proof}

 \begin{thm}\mlabel{thm:12.16} Let $A$ be a finite dimensional vector space. Then $(A,\mu,\Delta,\alpha,\beta,\psi,\omega)$ is a $\lambda$-infBH-bialgebra if and only if $(A^{\ast}, \Delta^{\ast}, \mu^{\ast}, \omega^{\ast}, \psi^{\ast}, \beta^{\ast}, \alpha^{\ast})$ is a $\lambda$-infBH-bialgebra with the multiplication
 \begin{eqnarray*}
 &A^{\ast}\otimes A^{\ast}\cong (A\otimes A)^{\ast}\stackrel{\Delta^{\ast}}\longrightarrow A^{\ast},&
 \end{eqnarray*}
 and the coproduct
 \begin{eqnarray*}
 &A^{\ast}\stackrel{\mu^{\ast}}\longrightarrow (A\otimes A)^{\ast}\cong A^{\ast}\otimes A^{\ast}.&
 \end{eqnarray*}
 \end{thm}

 \begin{proof} It can be proved by Proposition \mref{pro:12.15} and \cite[Theorem 5.5 and 5.6]{GMMP}.
 \end{proof}

\subsection{Two tensor product structures}

 \begin{lem}\mlabel{lem:12.4} (1) Let $(A, \mu, 1, \Delta, \alpha, \beta, \psi, \omega)$ be a unitary $\lambda$-infBH-bialgebra. Then
 \begin{eqnarray*}
 &\Delta(1)=-\lambda(1\otimes 1).&
 \end{eqnarray*}

 (2) Let $(A, \mu, \Delta, \varepsilon, \alpha, \beta, \psi, \omega)$ be a counitary $\lambda$-infBH-bialgebra. Then
 \begin{eqnarray*}
 &\varepsilon(ab)=-\lambda\varepsilon(a)\varepsilon(b),&
 \end{eqnarray*}
 for all $a, b \in A$.
 \end{lem}

 \begin{proof} (1) can be checked below.
 \begin{eqnarray*}
 \Delta(1)\hspace{-2mm}&\stackrel{(\mref{eq:12.4})}=&\hspace{-2mm}\omega(1)1_{1}\otimes\beta(1_{2})+\alpha(1_{1})\otimes 1_{2}\psi(1)+ \lambda(\alpha\omega(1)\otimes\beta\psi(1))\\
 &\stackrel{(\mref{eq:12.30})(\mref{eq:1.5})}=&\beta(1_{1})\otimes\beta(1_{2})+\alpha(1_{1})\otimes \alpha(1_{2})+ \lambda(1\otimes1)\\
 &\stackrel{(\mref{eq:12.2})}=&\Delta\circ\beta(1)+\Delta\circ\alpha(1)+\lambda(1\otimes1)\\
 &\stackrel{(\mref{eq:1.5})}=&2\Delta(1)+\lambda(1\otimes1).
 \end{eqnarray*}

 (2) Dual to (1).
 \end{proof}

 \begin{rmk} \mlabel{lem:12.5}
 The counit $\varepsilon$ of counitary $(-1)$-infBH-bialgebra is an algebra morphism. In this case, if the BiHom-associative algebra $(A,\mu,\alpha,\beta)$ has a unit $1$, then $\varepsilon(1)=1_{K}$.
 \end{rmk}

 Motivated by Lemma \mref{lem:12.4} (2), we can get an extended version below, which generalizes the augmented algebras introduced by Aguiar \cite{Ag04}.

 \begin{defi}\mlabel{de:12.6} Let $\lambda$ be a given element of $K$. A {\bf $\lambda$-augmented BiHom-associative algebra} is a 5-tuple $(A, \mu, \chi, \alpha, \beta)$ consisting of a BiHom-associative algebra $(A, \mu, \alpha, \beta)$ (possibly without unit) and a linear {\bf augmentation map $\chi:A\longrightarrow K$ of weight $\lambda$} satisfying $\chi\circ\alpha=\chi$, $\chi\circ\beta=\chi$ and
 \begin{eqnarray}
 &\chi(ab)=-\lambda\chi(a)\chi(b),&\mlabel{eq:12.5}
 \end{eqnarray}
 for all $a, b \in A$.

 Let $(A, \mu_{A}, \chi_{A}, \alpha_{A}, \beta_{A})$ and $(B, \mu_{B}, \chi_{B}, \alpha_{B}, \beta_{B})$ be two $\lambda$-augmented BiHom-associative algebras. An algebra map $f: A\longrightarrow B$ is said to be {\bf augmented} if it satisfies $\alpha_{B}\circ f=f\circ\alpha_{A}$, $\beta_{B}\circ f=f\circ\beta_{A}$ and $\chi_{B}\circ f=\chi_{A}$.
 \end{defi}

 \begin{rmk}\mlabel{rmk:12.8} By Lemma \mref{lem:12.4}, a counitary $\lambda$-infBH-bialgebra is a $\lambda$-augmented BiHom-associative algebra.
 \end{rmk}

 \begin{pro}\mlabel{pro:12.10} Let $(A, \mu_{A}, \chi_{A}, \alpha_{A}, \beta_{A})$ and $(B, \mu_{B}, \chi_{B}, \alpha_{B}, \beta_{B})$ be two $\lambda$-augmented BiHom-associative algebras. Then

 (1) $(A\otimes B, \cdot_{\chi}, \alpha_{A}\otimes\alpha_{B}, \beta_{A}\otimes\beta_{B})$ is a BiHom-associative algebra with the multiplication $\cdot_{\chi}$ defined by
 \begin{eqnarray}
 &&(a\otimes b)\cdot_{\chi}(a'\otimes b')\nonumber\\
 &&\qquad =\chi_{B}(b)aa'\otimes\beta_{B}(b')+\chi_{A}(a')\alpha_{A}(a)\otimes bb'+\lambda\chi_{A}(a')\chi_{B}(b)\alpha_{A}(a)\otimes\beta_{B}(b'), \mlabel{eq:12.6}
 \end{eqnarray}
 for all $a,a'\in A$ and $b,b'\in B$.

 (2) Furthermore, $(A\otimes B, \cdot_{\chi}, \chi_{A\otimes B}, \alpha_{A}\otimes\alpha_{B}, \beta_{A}\otimes\beta_{B})$ is a $\lambda$-augmented BiHom-associative algebra with the augmentation map given by
 \begin{eqnarray}
 &\chi_{A\otimes B}=\chi_{A}\o \chi_{B}.&\mlabel{eq:12.7}
 \end{eqnarray}
 \end{pro}

 \begin{proof} (1) We only check the BiHom-associativity as follows. For all $a,a',a''\in A$ and $b,b',b''\in B$, on one hand,
 \begin{eqnarray*}
 &&\hspace{-20mm}((a\otimes b)\cdot_{\chi}(a'\otimes b'))\cdot_{\chi}(\beta_{A}(a'')\otimes \beta_{B}(b''))\\
 &\stackrel{(\mref{eq:12.6})}=&\chi_{B}(b)(\chi_{B}(b')(aa')\beta_{A}(a'')
 \otimes\beta^{2}_{B}(b'')
 +\chi_{A}(a'')\alpha_{A}(aa')\otimes\beta_{B}(b')\beta_{B}(b'')\\
 &&+\lambda\chi_{A}(a'')\chi_{B}(b')\alpha_{A}(aa')\otimes\beta^{2}_{B}(b''))
 +\chi_{A}(a')(\chi_{B}(bb')\alpha_{A}(a)\beta_{A}(a'')\otimes\beta^{2}_{B}(b'')\\
 &&+\chi_{A}(a'')\alpha^{2}_{A}(a)\otimes(bb')\beta_{B}(b'')
 +\lambda\chi_{A}(a'')\chi_{B}(bb')\alpha^{2}_{A}(a)\otimes\beta^{2}_{B}(b''))\\
 &&+\lambda\chi_{A}(a')\chi_{B}(b)(\chi_{B}(b')\alpha_{A}(a)\beta_{A}(a'')
 \otimes\beta^{2}_{B}(b'')
 +\chi_{A}(a'')\alpha^{2}_{A}(a)\otimes\beta_{B}(b')\beta_{B}(b'')\\
 &&+\lambda\chi_{A}(a'')\chi_{B}(b')\alpha^{2}_{A}(a)\otimes\beta^{2}_{B}(b''))\\
 &\stackrel{(\mref{eq:1.3})(\mref{eq:12.5})}=&\chi_{B}(b)\chi_{B}(b')\alpha_{A}(a)(a'a'')
 \otimes\beta^{2}_{B}(b'')
 +\chi_{B}(b)\chi_{A}(a'')\alpha_{A}(a)\alpha_{A}(a')\otimes\beta_{B}(b')\beta_{B}(b'')\\
 &&+\lambda\chi_{B}(b)\chi_{A}(a'')\chi_{B}(b')\alpha_{A}(a)\alpha_{A}(a')
 \otimes\beta^{2}_{B}(b'')
 -\lambda\chi_{A}(a')\chi_{B}(b)\chi_{B}(b')\alpha_{A}(a)\beta_{A}(a'')\\
 &&\otimes\beta^{2}_{B}(b'')+\chi_{A}(a')\chi_{A}(a'')\alpha^{2}_{A}(a)\otimes\alpha_{B}(b)(b'b'')
 -\lambda^{2}\chi_{A}(a')\chi_{A}(a'')\chi_{B}(b)\chi_{B}(b')
 \alpha^{2}_{A}(a)\\
 &&\otimes\beta^{2}_{B}(b'')+\lambda\chi_{A}(a')\chi_{B}(b)\chi_{B}(b')\alpha_{A}(a)\beta_{A}(a'')
 \otimes\beta^{2}_{B}(b'')
 +\lambda\chi_{A}(a')\chi_{B}(b)\chi_{A}(a'')\alpha^{2}_{A}(a)\\
 &&\otimes\beta_{B}(b')\beta_{B}(b'')+\lambda^{2}\chi_{A}(a')\chi_{B}(b)\chi_{A}(a'')\chi_{B}(b')
 \alpha^{2}_{A}(a)\otimes\beta^{2}_{B}(b'')\\
 &=&\chi_{B}(b)\chi_{B}(b')\alpha_{A}(a)(a'a'')
 \otimes\beta^{2}_{B}(b'')
 +\chi_{B}(b)\chi_{A}(a'')\alpha_{A}(a)\alpha_{A}(a')\otimes\beta_{B}(b')\beta_{B}(b'')\\
 &&+\lambda\chi_{B}(b)\chi_{A}(a'')\chi_{B}(b')\alpha_{A}(a)\alpha_{A}(a')
 \otimes\beta^{2}_{B}(b'')
 +\chi_{A}(a')\chi_{A}(a'')\alpha^{2}_{A}(a)\otimes\alpha_{B}(b)(b'b'')\\
 &&+\lambda\chi_{A}(a')\chi_{B}(b)\chi_{A}(a'')\alpha^{2}_{A}(a)
 \otimes\beta_{B}(b')\beta_{B}(b'')\stackrel{\bigtriangleup}=I.
 \end{eqnarray*}
 On the other hand,
 \begin{eqnarray*}
 &&\hspace{-15mm}(\alpha_{A}(a)\otimes \alpha_{B}(b))\cdot_{\chi}((a'\otimes b')\cdot_{\chi}(a''\otimes b''))\\
 &\stackrel{(\mref{eq:12.6})}=&\chi_{B}(b')(\chi_{B}(b)\alpha_{A}(a)(a'a'')
 \otimes\beta^{2}_{B}(b'')
 +\chi_{A}(a'a'')\alpha^{2}_{A}(a)\otimes\alpha_{B}(b)\beta_{B}(b'')\\
 &&+\lambda\chi_{A}(a'a'')\chi_{B}(b)\alpha^{2}_{A}(a)\otimes\beta^{2}_{B}(b''))
 +\chi_{A}(a'')(\chi_{B}(b)\alpha_{A}(a)\alpha_{A}(a')\otimes\beta_{B}(b'b'')\\
 &&+\chi_{A}(a')\alpha^{2}_{A}(a)\otimes\alpha_{B}(b)(b'b'')
 +\lambda\chi_{A}(a')\chi_{B}(b)\alpha^{2}_{A}(a)\otimes\beta_{B}(b'b''))\\
 &&+\lambda\chi_{A}(a'')\chi_{B}(b')(\chi_{B}(b)\alpha_{A}(a)\alpha_{A}(a')
 \otimes\beta^{2}_{B}(b'')
 +\chi_{A}(a')\alpha^{2}_{A}(a)\otimes\alpha_{B}(b)\beta_{B}(b'')\\
 &&+\lambda\chi_{A}(a')\chi_{B}(b)\alpha^{2}_{A}(a)\otimes\beta^{2}_{B}(b''))\\
 &\stackrel{(\mref{eq:12.5})}=&\chi_{B}(b')\chi_{B}(b)\alpha_{A}(a)(a'a'')
 \otimes\beta^{2}_{B}(b'')
 -\lambda\chi_{B}(b')\chi_{A}(a')\chi_{A}(a'')\alpha^{2}_{A}(a)
 \otimes\alpha_{B}(b)\beta_{B}(b'')\\
 &&-\lambda^{2}\chi_{B}(b')\chi_{A}(a')\chi_{A}(a'')\chi_{B}(b)
 \alpha^{2}_{A}(a)\otimes\beta^{2}_{B}(b'')
 +\chi_{A}(a'')\chi_{B}(b)\alpha_{A}(a)\alpha_{A}(a')\otimes\beta_{B}(b')\beta_{B}(b'')\\
 &&+\chi_{A}(a'')\chi_{A}(a')\alpha^{2}_{A}(a)\otimes\alpha_{B}(b)(b'b'')
 +\lambda\chi_{A}(a'')\chi_{A}(a')\chi_{B}(b)\alpha^{2}_{A}(a)
 \otimes\beta_{B}(b')\beta_{B}(b')\\
 &&+\lambda\chi_{A}(a'')\chi_{B}(b')\chi_{B}(b)\alpha_{A}(a)\alpha_{A}(a')
 \otimes\beta^{2}_{B}(b'')
 +\lambda\chi_{A}(a'')\chi_{B}(b')\chi_{A}(a')\alpha^{2}_{A}(a)
 \otimes\alpha_{B}(b)\beta(b'')\\
 &&+\lambda^{2}\chi_{A}(a'')\chi_{B}(b')\chi_{A}(a')\chi_{B}(b)\alpha^{2}_{A}(a)
 \otimes\beta^{2}_{B}(b'')=I,
 \end{eqnarray*}
 as desired.

 (2) For all $a, a'\in A$ and $b, b'\in B$,
 \begin{eqnarray*}
 &&\hspace{-15mm}\chi_{A\otimes B}((a\otimes b)\cdot_{\chi}(a'\otimes b'))\\
 &\stackrel{(\mref{eq:12.6})}=&\chi_{A\otimes B}(\chi_{B}(b)aa'\otimes\beta_{B}(b')+\chi_{A}(a')\alpha_{A}(a)\otimes bb'+\lambda\chi_{A}(a')\chi_{B}(b)\alpha_{A}(a)\otimes\beta_{B}(b'))\\
 &\stackrel{(\mref{eq:12.7})}=&\chi_{B}(b)\chi_{A}(aa')\chi_{B}(b')
 +\chi_{A}(a')\chi_{A}(a)\chi_{B} (bb')+\lambda\chi_{A}(a')\chi_{B}(b)\chi_{A}(a)\chi_{B}(b')\\
 &\stackrel{(\mref{eq:12.5})}=&
 -\lambda\chi_{A}(a)\chi_{B} (b)\chi_{A}(a')\chi_{B}(b')\\
 &\stackrel{(\mref{eq:12.7})}=&-\lambda\chi_{A\otimes B}(a\otimes b)\chi_{A\otimes B}(a'\otimes b'),
 \end{eqnarray*}
 finishing the proof.
 \end{proof}

 Dual to Definition \mref{de:12.6}, we have

 \begin{defi}\mlabel{de:12.46} Let $\lambda$ be a given element of $K$. A {\bf $\lambda$-coaugmented BiHom-coassociative coalgebra} is a 5-tuple $(A, \Delta, \zeta, \psi, \omega)$ consisting of a BiHom-coassociative coalgebra $(A, \Delta, \psi, \omega)$ (possibly without counit) and a linear {\bf coaugmentation map $\zeta:K\longrightarrow A$ of weight $\lambda$} satisfying $\omega\circ\zeta=\zeta$, $\psi\circ\zeta=\zeta$ and
 \begin{eqnarray}
 &\Delta\ci \zeta=-\lambda \zeta\otimes \zeta.&\mlabel{eq:12.42}
 \end{eqnarray}
 \end{defi}

 \begin{rmk}\mlabel{rmk:12.8a}
  By Lemma \mref{lem:12.4}, a unitary $\lambda$-infBH-bialgebra is a $\lambda$-coaugmented BiHom-coassociative coalgebra with the coaugmentation map $\eta: K\lr A$ ($\eta(1_K)=1_A$).
 \end{rmk}

 \begin{pro}\mlabel{pro:12.47} Let $(C, \Delta_{C}, \zeta_{C}, \psi_{C}, \omega_{C})$ and $(D, \Delta_{D}, \zeta_{D}, \psi_{D}, \omega_{D})$ be two $\lambda$-coaugmented BiHom-coassociative coalgebras. Then

 (1) $(C\otimes D, \Delta_{\zeta}, \psi_{C}\otimes\psi_{D}, \omega_{C}\otimes\omega_{D})$ is a BiHom-coassociative coalgebra with the coproduct defined by
 \begin{eqnarray}
 &&\Delta_{\zeta}(c\otimes d)=(c_{1}\otimes \mathcal{I}_{D})\otimes(c_{2}\otimes \psi_{D}(d))+(\omega_{C}(c)\otimes d_{1})\otimes (\mathcal{I}_{C}\otimes d_{2})\nonumber\\
 &&\qquad\qquad\qquad +\lambda(\omega_{C}(c)\otimes \mathcal{I}_{D})\otimes(\mathcal{I}_{C}\otimes\psi_{D}(d)). \mlabel{eq:12.8}
 \end{eqnarray}
 for all $c\in C$, $d\in D$ and $\zeta_C(1_K)=\mathcal{I}_{C}, \zeta_D(1_K)=\mathcal{I}_{D}$.

 (2) Furthermore, $(C\otimes D, \Delta, \zeta_{C\otimes D}, \psi_{C}\otimes\psi_{D}, \omega_{C}\otimes\omega_{D})$ is a $\lambda$-coaugmented BiHom-coassociative coalgebra with the coaugmentation map given by
 \begin{eqnarray}
 &\zeta_{C\otimes D}=\zeta_{C}\otimes \zeta_{D}.&\mlabel{eq:12.43}
 \end{eqnarray}
 \end{pro}

 \begin{proof} Dual to the proof of Proposition \mref{pro:12.10}. \end{proof}

 \begin{rmk} Proposition \mref{pro:12.10} (\mref{pro:12.47}) shows that the tensor product of two $\lambda$-(co)augmented BiHom-(co)associative (co)algebras is closed, which implies that the category of $\lambda$-(co)augmented BiHom-(co)associative (co)algebras is  a tensor category.
 \end{rmk}

 Based on the new (co)algebra structures on the tensor product of two (co)algebras in Proposition \mref{pro:12.10} (\mref{pro:12.47}), we have the following results but the usual tensor product (co)algebra structures do not work.

 \begin{thm}\mlabel{thm:12.12} (1) Let $(A,\mu, \Delta, \varepsilon, \alpha, \beta, \psi, \omega)$ be a counitary $\lambda$-infBH-bialgebra and view $(A\otimes A, \cdot_{\varepsilon}, \alpha\otimes \alpha, \beta\otimes\beta)$ as a BiHom-associative algebra as in Proposition \mref{pro:12.10}. Then $\Delta: (A,\mu,\alpha,\beta)\longrightarrow (A\otimes A,\cdot_{\varepsilon},\alpha\otimes\alpha,\beta\otimes\beta)$ is a morphism of BiHom-associative algebras.

 (2) Let $(C,\mu,\Delta,1,\alpha,\beta,\psi,\omega)$ be a unitary $\lambda$-infBH-bialgebra and view $(C\otimes C,\Delta_{\eta},\psi\otimes\psi,\omega\otimes\omega)$ as a BiHom-coassociative coalgebra as in Proposition \mref{pro:12.47}. Then $\mu: (C\otimes C,\Delta_{\eta},\psi\otimes\psi,\omega\otimes\omega)\longrightarrow (C,\Delta,\psi,\omega)$ is a morphism of BiHom-coassociative coalgebras.
 \end{thm}

 \begin{proof} (1) We only need to prove that $\Delta(a a')=\Delta(a)\cdot_{\varepsilon}\Delta(a')$ for all $a, a'\in A$. In fact, we have
 \begin{eqnarray*}
 \Delta(a)\cdot_{\varepsilon}\Delta(a')
 &=&(a_{1}\otimes a_{2})\cdot_{\varepsilon}(a'_{1}\otimes a'_{2})\\
 &\stackrel{(\mref{eq:12.6})}=&\varepsilon(a_{2})a_{1}a'_{1}\otimes\beta(a'_{2})
 +\varepsilon(a'_{1})\alpha(a_{1})\otimes a_{2}a'_{2}+\lambda\varepsilon(a'_{1})\varepsilon(a_{2})\alpha(a_{1})\otimes\beta(a'_{2})\\
 &\stackrel{(\mref{eq:1.11})}=&\omega(a)a'_{1}\otimes\beta(a'_{2})
 +\alpha(a_{1})\otimes a_{2}\psi(a')+\lambda\alpha\omega(a)\otimes\beta\psi(a')\\
 &\stackrel{(\mref{eq:12.4})}=&\Delta(a a').
 \end{eqnarray*}

 (2) Dual to (1).
 \end{proof}

\section{$\l$-infBH-Hopf modules}\label{se:infbhmod}
 In order to study  representations of $\l$-infBH-bialgebras, In this section, we introduce the notion of $\l$-infBH-Hopf modules and prove that  (co)modules of (co)quasitriangular $\l$-infBH-bialgebras can induce  structures of $\l$-infBH-Hopf modules, which are similar to the classical Hopf algebra theory.

\subsection{Definition and examples} Based on the notion of $\lambda$-infBH-bialgebra, we get the following natural definition.
 \begin{defi}\mlabel{de:12.17} Let $\lambda$ be a given element of $K$ and $(A,\mu,\Delta,\alpha_{A},\beta_{A},\psi_{A},\omega_{A})$ be a $\lambda$-infBH-bialgebra. A {\bf (left) $\lambda$-infBH-Hopf module} is a 7-tuple $(M,\gamma,\rho, \alpha_{M},\beta_{M},$ $\psi_{M},\omega_{M})$, where $(M, \gamma, \alpha_{M}$, $\beta_{M})$ is a left $(A,\mu,\alpha_{A},\beta_{A})$-module and $(M, \rho, \psi_{M}, \omega_{M})$ is a left $(A,\Delta,\psi_{A},\omega_{A})$-comodule, such that any two maps of $\alpha_{M}, \beta_{M}, \psi_{M}, \omega_{M}$ commute and
 \begin{eqnarray}
 &\rho\gamma=(\mu\otimes \beta_{M})(\omega_{A}\otimes\rho)+(\alpha_{A}\otimes\gamma)(\Delta\otimes\psi_{M})
 +\l\alpha_{A}\omega_{A}\otimes\beta_{M}\psi_{M}.&\mlabel{eq:12.13}
 \end{eqnarray}

 If further $(A,\mu,1,\Delta,\alpha_{A},\beta_{A},\psi_{A},\omega_{A})$ ($(A,\mu,\v,\Delta,\alpha_{A},\beta_{A},\psi_{A},\omega_{A})$) is a (co)unitary $\lambda$-infBH-bialgebra, then $(M,\gamma,\rho,$ $\alpha_{M},\beta_{M},\psi_{M},\omega_{M})$ is called a {\bf (co)unitary (left) $\lambda$-infBH-Hopf module}.
 \end{defi}

 \begin{rmk}\mlabel{rmk:12.18} (1) In terms of above notations, the compatibility condition of left $\l$-infBH-Hopf module given in Eq.(\mref{eq:12.13}) may be written as
 \begin{eqnarray*}
 &(a\tl m)_{-1}\otimes (a\tl m)_{0}=\omega_{A}(a)m_{-1}\otimes\beta_{M}(m_{0})+\alpha_{A}(a_{1})\otimes a_{2}\tl\psi_{M}(m)+
 \lambda\alpha_{A}\omega_{A}(a)\otimes\beta_{M}\psi_{M}(m),&
 \end{eqnarray*}
 here we write $\g(a\o m)=a\tl m$ and $\rho(m)=m_{-1}\o m_{0}$.

 (2) The right version can be given similarly, which will be used in the forthcoming paper \cite{MM} to construct a $\lambda$-infBH-Hopf bimodule over a $\lambda$-infBH-bialgebra. In this paper, a $\l$-infBH-Hopf module means the left one.
 \end{rmk}

We provide in the following some relevant examples.

 \begin{ex}\mlabel{ex:12.19} Let $\lambda$ be a given element of $K$. (1) Let $(A, \mu, \Delta, \alpha_A, \beta_A, \psi_A, \omega_A)$ be a $\lambda$-infBH-bialgebra. Obviously, $(A, \mu, \Delta, \alpha_A, \beta_A, \psi_A, \omega_A)$ itself is a $\lambda$-infBH-Hopf module by taking $\gamma=\mu$ and $\rho=\Delta$ ($\nu=\mu$ and $\vp=\Delta$).

 (2) Let $(A, \mu, \Delta, \alpha_A, \beta_A, \psi_A, \omega_A)$ be a $\lambda$-infBH-bialgebra, $V$ be a vector space and $\a_V, \b_V, \psi_V, \om_V: V\lr V$ be four linear maps such that any two of them commute. Then $(A\otimes V, \gamma, \rho, \alpha_{A}\otimes\alpha_{V}, \beta_{A}\otimes\beta_{V},
 \psi_{A}\otimes\psi_{V}, \omega_{A}\otimes\omega_{V})$ is a $\lambda$-infBH-Hopf module with the structure maps
 \begin{eqnarray*}
 &\gamma=\mu\otimes\beta_{V}:A\otimes A\otimes V\longrightarrow A\otimes V,\ \rho=\Delta\otimes\psi_{V}:A\otimes V\longrightarrow A\otimes A\otimes V.&
 \end{eqnarray*}

 We have the following two more general versions than the above example.

 (3) Let $(A, \mu, 1, \Delta, \alpha_A, \beta_A, \psi_A, \omega_A)$ be a unitary $\l$-infBH-bialgebra, $V$ be a vector space, $\alpha_{V}$, $\beta_{V}$, $\psi_{V}$, $\omega_{V}: V\longrightarrow V$ be four linear maps such that any two of them commute. Then $(A\otimes V, \gamma, \rho, \alpha_{A}\otimes\alpha_{V}, \beta_{A}\otimes\beta_{V},
 \psi_{A}\otimes\psi_{V}, \omega_{A}\otimes\omega_{V})$ is a $\lambda$-infBH-Hopf module with the structure maps
 \begin{eqnarray*}
 &&\gamma:A\otimes (A\otimes V)\longrightarrow A\otimes V\\
 &&~~~~\quad a\otimes (b\otimes n)\longmapsto ab\otimes\beta_{V}(v)
 \end{eqnarray*}
 and
 \begin{eqnarray*}
 &&\rho: A\otimes V\longrightarrow A\otimes (A\otimes V)\\
 &&~~~\quad a\otimes v\longmapsto \Delta(a)\otimes \psi_{V}(v)+\lambda\omega_{A}(a) \otimes 1_{A}\otimes \psi_{V}(v)
 \end{eqnarray*}
 for all $a,b \in A$ and $v\in V$.

 \begin{proof} In fact, for all $a, b, c\in A$ and $v\in V$, we have
 \begin{eqnarray*}
 \gamma(\alpha_{A}\otimes\gamma)(a\otimes b\otimes c\otimes v)
 &\stackrel{(\mref{eq:1.3})}=&(ab)\beta_{A}(c)\otimes \beta^{2}_{V}(v)\\
 &=&\gamma(\mu\otimes\beta_{A\otimes V})(a\otimes b\otimes c\otimes v),
 \end{eqnarray*}
 \begin{eqnarray*}
 (\omega_{A}\otimes\rho)\rho (a\otimes v)
 &=&\omega_{A}(a_{1})\otimes a_{21}\otimes a_{22}\otimes\psi_{V}^{2}(v)
 +\lambda\omega_{A}(a_{1})\otimes\omega_{A}(a_{2})\otimes 1_{A}\otimes\psi_{V}^{2}(v)\\
 &&+\lambda\omega_{A}^{2}(a)\otimes 1_{A1}\otimes 1_{A2}\otimes\psi_{V}^{2}(v)
 +\lambda^{2}\omega_{A}^{2}(a)\otimes\omega_{A}(1_{A})\otimes 1_{A}\otimes\psi_{V}^{2}(v)\\
 &\stackrel{(\mref{eq:1.9})(\mref{eq:12.30})}=&a_{11}\otimes a_{12}\otimes \psi_{A}(a_{2})\otimes\psi_{V}^{2}(v)
 +\lambda\omega_{A}(a_{1})\otimes\omega_{A}(a_{2})\otimes 1_{A}\otimes\psi_{V}^{2}(v)\\
 &&\hspace{60mm}\quad(\hbox{also by Lemma \mref{lem:12.4}})\\
 &\stackrel{(\mref{eq:1.7})(\mref{eq:12.30})}=&(\Delta\otimes\psi_{A\otimes V})\rho(a\otimes v).
 \end{eqnarray*}
 Thus $(A\otimes V, \gamma, \alpha_{A}\otimes\alpha_{V}, \beta_{A}\otimes\beta_{V})$ is a left $(A, \mu, \a_A, \b_A)$-module and $(A\otimes V, \rho, \psi_{A}\otimes\psi_{V}, \omega_{A}\otimes\omega_{V})$ is a left $(A, \D, \psi, \om)$-comodule. The compatibility condition can be checked as follows:
 \begin{eqnarray*}
 \rho\gamma(a\otimes(b\otimes v))
 &=&(ab)_{1}\otimes (ab)_{2}\otimes\psi_{V}\beta_{V}(v)+\lambda\omega_{A}(ab)\otimes 1_{A}\otimes\psi_{V}\beta_{V}(v)\\
 &\stackrel{(\mref{eq:12.4})(\mref{eq:12.3})}=&\omega_{A}(a)b_{1}\otimes \beta_{A}(b_{2})\otimes \psi_{V}\beta_{V}(v)+\lambda\omega_{A}(a)\omega_{A}(b)\otimes 1_{A}\otimes \psi_{V}\beta_{V}(v)\\
 &&+\alpha_{A}(a_{1})\otimes a_{2}\psi_{A}(b)\otimes\psi_{V}\beta_{V}(v)
 +\lambda\alpha_{A}\omega_{A}(a)\otimes\beta_{A}\psi_{A}(b)\otimes\psi_{V}\beta_{V}(v)\\
 &\stackrel{(\mref{eq:1.5})}=&(\mu\otimes \beta_{A\otimes V})(\omega_{A}\otimes\rho)(a\otimes(b\otimes v))+(\alpha_{A}\otimes\gamma)(\Delta\otimes\psi_{A\otimes V})(a\otimes(b\otimes v)),
 \end{eqnarray*}
 completing the proof.
 \end{proof}

 (4) Let $(A, \mu, \Delta, \varepsilon, \alpha_{A}, \beta_{A}, \psi_{A}, \omega_{A})$ be a counitary $\lambda$-infBH-bialgebra, $V$  be a vector space, and $\alpha_{V}, \beta_{V},\psi_{V},\omega_{V}: V\longrightarrow V$ be four linear maps such that any two of them commute. Then $(A\otimes V, \gamma, \rho, \alpha_{A}\otimes\alpha_{V}, \beta_{A}\otimes\beta_{V},
 \psi_{A}\otimes\psi_{V}, \omega_{A}\otimes\omega_{V})$ is a $\lambda$-infBH-Hopf module with the structure maps
 \begin{eqnarray*}
 &&\gamma: A\otimes (A\otimes V)\longrightarrow A\otimes V\\
 &&~~~\quad a\otimes (b\otimes v)\longmapsto ab\otimes\beta_{V}(v) +\lambda\varepsilon(b)\alpha_{A}(a)\otimes \beta_{V}(v)
 \end{eqnarray*}
 and
 \begin{eqnarray*}
 &&\rho: A\otimes V\longrightarrow A\otimes (A\otimes V)\\
 &&~~~\quad a\otimes v\longmapsto a_1\o a_2\otimes \psi_{V}(v)
 \end{eqnarray*}
 for all $a,b \in A$ and $v\in V$.

 \begin{proof} By the proof of Example \mref{ex:12.19} (3), we know that $(A\o V, \rho, \psi_{A}\o \psi_V, \omega_{A}\o \om_V)$ is a left $(A,\Delta,\psi_{A},\omega_{A})$-comodule. For all $a,b,c\in A$ and $v\in V$, we have
 \begin{eqnarray*}
 \alpha_{A}(a)\triangleright(b\triangleright (c\otimes v))
 &\stackrel{(\mref{eq:12.31})}=&\alpha_{A}(a)(bc)\otimes\beta^{2}_{V}(v)
 +\lambda\varepsilon(bc)\alpha_{A}^{2}\otimes\beta_{V}^{2}(v)\\
 &&+\lambda\varepsilon(c)\alpha_{A}(a)\alpha_{A}(b)\otimes\beta_{V}^{2}(v)
 +\lambda^{2}\varepsilon(b)\varepsilon(c)\alpha_{A}^{2}(a) \otimes\beta_{V}^{2}(v)\\
 &\stackrel{(\mref{eq:1.3})}=&(ab)\beta_{A}(c)\otimes\beta^{2}_{V}(v)
 -\lambda^{2}\varepsilon(b)\varepsilon(c)\alpha_{A}^{2}(a)\otimes\beta_{V}^{2}(v)\\
 &&+\lambda\varepsilon(c)\alpha_{A}(a)\alpha_{A}(b)\otimes\beta_{V}^{2}(v)
 +\lambda^{2}\varepsilon(b)\varepsilon(c)\alpha_{A}^{2}(a) \otimes\beta_{V}^{2}(v)~(\hbox{by Lemma \mref{lem:12.4}})\\
 &\stackrel{(\mref{eq:1.2})}=&(ab)\beta_{A}(c)\otimes\beta^{2}_{V}(v)
 +\lambda\varepsilon(c)\alpha_{A}(ab)\otimes \beta_{V}^{2}(v)\\
 &\stackrel{(\mref{eq:12.31})}=&ab\triangleright \beta_{A\otimes V}(c\otimes v).
 \end{eqnarray*}
 Thus $(A\o V, \gamma, \alpha_{A}\o \a_V, \beta_{A}\o \b_V)$ is a left $(A,\mu,\alpha_{A},\beta_{A})$-module. Then it remains to check the compatibility condition of the left $\lambda$-infBH-Hopf module.
 \begin{eqnarray*}
 \rho\gamma(a\otimes(b\otimes v))
 &=&\Delta(ab)\otimes\beta_{V}\psi_{V}(v)+\lambda\varepsilon(b)\alpha_{A}(a)_{1}
 \otimes\alpha_{A}(a)_{2}\otimes \beta_{V}\psi_{V}(v)\\
 &\stackrel{(\mref{eq:12.4})(\mref{eq:12.2})}=&\omega_{A}(a)b_{1}\otimes\beta_{A}(b_{2})\otimes\beta_{V}\psi_{V}(v)
 +\alpha_{A}(a_{1})\otimes a_{2}\psi_{A}(b)\otimes \beta_{V}\psi_{V}(v)\\
 &&+\lambda\varepsilon(b)\alpha_{A}(a_{1})\otimes\alpha_{A}(a_{2})\otimes \psi_{V}\beta_{V}(v)
 +\lambda\alpha_{A}\omega_{A}(a)\otimes\beta_{A}\psi_{A}(b)\otimes\beta_{V}\psi_{V}(v)\\
 &\stackrel{(\mref{eq:1.11})}=&(\mu\otimes \beta_{A\otimes V})(\omega_{A}\otimes\rho)(a\otimes(b\otimes v))+(\alpha_{A}\otimes\gamma)(\Delta\otimes\psi_{A\otimes V})(a\otimes(b\otimes v)).
 \end{eqnarray*}
 These finish the proof since other conditions are obviously satisfied.
 \end{proof}

 (5) Let $(A, \mu, \Delta, 1_{A}, \alpha_{A}, \beta_{A}, \psi_{A}, \omega_{A})$ be a unitary $0$-infBH-bialgebra such that $\a_A$ is invertible and $(N, \rho, \psi_{N}, \omega_{N})$ be a left $(A,\Delta,\psi_{A},\omega_{A})$-comodule, $\alpha_{N},\beta_{N}: N\longrightarrow N$ be linear maps such that any two maps of $\alpha_{N}, \beta_{N}, \psi_{N}, \omega_{N}$ commute. Then the space $A\otimes N$ is a $0$-infBH-Hopf module with the structure maps
 \begin{eqnarray*}
 &&\gamma:A\otimes (A\otimes N)\longrightarrow A\otimes N\\
 &&~~~~\quad a\otimes (b\otimes n)\longmapsto ab\otimes\beta_{N}(n)
 \end{eqnarray*}
 and
 \begin{eqnarray*}
 &&\rho: A\otimes N\longrightarrow A\otimes (A\otimes N)\\
 &&~~~\quad a\otimes n\longmapsto \Delta(a)\otimes \psi_{N}(n)+\omega_{A}\alpha^{-1}_{A}(a)n_{-1}\otimes 1_{A}\otimes n_{0}
 \end{eqnarray*}
 for all $a,b \in A$ and $n\in N$.

 \begin{proof} For all $a,b\in A$, $n\in N$, we have
 \begin{eqnarray*}
 (\omega_{A}\otimes\rho)\rho (a\otimes n)
 &\stackrel{(\mref{eq:1.5})}=&\omega_{A}(a_{1})\otimes a_{21}\otimes a_{22}\otimes\psi_{N}^{2}(n)
 +\omega_{A}(a_{1})\otimes\omega_{A}\alpha_{A}^{-1}(a_{2})\psi_{N}(n)_{-1}\otimes 1_{A}\otimes \psi_{N}(n)_{0}\\
 &&+\omega_{A}^{2}\alpha_{A}^{-1}(a)\omega_{A}(n_{-1})\otimes 1_{A1}\otimes 1_{A2}\otimes\psi_{N}(n_{0})
 +\omega_{A}^{2}\alpha_{A}^{-1}(a)\omega_{A}(n_{-1})\\
 &&\otimes \omega_{A}\alpha_{A}^{-1}(1_{A})n_{0-1}\otimes 1_{A}\otimes n_{00}\\
 &\stackrel{(\mref{eq:1.9})}=&a_{11}\otimes a_{12}\otimes \psi_{A}(a_{2})\otimes\psi_{N}^{2}(n)
 +\omega_{A}^{2}\alpha_{A}^{-1}(a)n_{-11}\otimes\beta_{A}(n_{-12})\otimes 1_{A}\otimes \psi_{N}(n_{0})\\
 &&+\omega_{A}(a_{1})\otimes \omega_{A}\alpha_{A}^{-1}(a_{2})\psi_{A}(n_{-1})\otimes 1_{A}\otimes \psi_{N}(n_{0})\quad(\hbox{by Lemma \mref{lem:12.4}})\\
 &\stackrel{(\mref{eq:12.4})}=&(\Delta\otimes\psi_{A\otimes N})\rho(a\otimes n),
 \end{eqnarray*}
 and
 \begin{eqnarray*}
 \rho\gamma(a\otimes(b\otimes n))
 &=&(ab)_{1}\otimes (ab)_{2}\otimes\psi_{N}\beta_{N}(n)+\omega_{A}\alpha_{A}^{-1}(ab)\beta_{A}(n_{-1})\otimes 1_{A}\otimes\beta_{N}(n_{0})\\
 &\stackrel{(\mref{eq:12.4})}=&\omega_{A}(a)b_{1}\otimes \beta_{A}(b_{2})\otimes \psi_{N}\beta_{N}(n)+\alpha_{A}(a_{1})\otimes a_{2}\psi_{A}(b)\otimes \psi_{N}\beta_{N}(n)\\
 &&+\omega_{A}\alpha_{A}^{-1}(ab)\beta_{A}(n_{-1})\otimes 1_{A}\otimes\beta_{N}(n_{0})\\
 &\stackrel{(\mref{eq:1.3})}=&\omega_{A}(a)b_{1}\otimes \beta_{A}(b_{2})\otimes \psi_{N}\beta_{N}(n)+\omega_{A}(a)(\omega_{A}\alpha_{A}^{-1}(b)n_{-1})\otimes 1_{A}\otimes\beta_{N}(n_{0})\\
 &&+\alpha_{A}(a_{1})\otimes a_{2}\psi_{A}(b)\otimes \psi_{N}\beta_{N}(n)\\
 &=&(\mu\otimes \beta_{A\otimes N})(\omega_{A}\otimes\rho)(a\otimes(b\otimes n))+(\alpha_{A}\otimes\gamma)(\Delta\otimes\psi_{A\otimes N})(a\otimes(b\otimes n)).
 \end{eqnarray*}
 The rest is obvious by Example \mref{ex:12.19} (3). These complete the proof. \end{proof}

 (6) Let $(A, \mu, \Delta, \varepsilon, \alpha_{A}, \beta_{A}, \psi_{A}, \omega_{A})$ be a counitary $0$-infBH-bialgebra such that $\omega_A$ is invertible and $(N,\triangleright,\alpha_{N},\beta_{N})$ be a left $(A,\mu,\alpha_{A},\beta_{A})$-module, $\psi_{N},\omega_{N}: N\longrightarrow N$ be linear maps such that any two maps of $\alpha_{N}, \beta_{N}, \psi_{N}, \omega_{N}$ commute. Then the space $A\otimes N$ is a $0$-infBH-Hopf module with the structure maps
 \begin{eqnarray*}
 &&\gamma: A\otimes (A\otimes N)\longrightarrow A\otimes N\\
 &&~~~\quad a\otimes (b\otimes n)\longmapsto ab\otimes\beta_{N}(n) +\varepsilon(b)\alpha_{A}\omega^{-1}_{A}(a_{1})\otimes (a_{2}\triangleright n)
 \end{eqnarray*}
 and
 \begin{eqnarray*}
 &&\rho: A\otimes N\longrightarrow A\otimes (A\otimes N)\\
 &&~~~\quad a\otimes n\longmapsto a_1\o a_2\otimes \psi_{N}(n)
 \end{eqnarray*}
 for all $a,b \in A$ and $n\in N$.

 \begin{proof} By Example \mref{ex:12.19} (4), we know that $(A\o N, \rho, \psi_{A}\o \psi_N, \omega_{A}\o \om_N)$ is a left $(A,\Delta,\psi_{A},\omega_{A})$-comodule. For any $a,b,c\in A$ and $n\in N$, we have
 \begin{eqnarray*}
 \alpha_{A\otimes N}(a\triangleright(b\otimes n))
 &=&\alpha_{A}(ab)\otimes \alpha_{N}\beta_{N}(n)+\varepsilon(b)\alpha^{2}_{A}\omega^{-1}_{A}(a_{1})\otimes \alpha_{N}(a_{2}\triangleright n)\\
 &\stackrel{(\mref{eq:1.2})(\mref{eq:1.13})}=&\alpha_{A}(a)\alpha_{A}(b)\otimes \alpha_{N}\beta_{N}(n)+\varepsilon(b)\alpha^{2}_{A}\omega^{-1}_{A}(a_{1})\otimes (\alpha_{A}(a_{2})\triangleright \alpha_{N}(n))\\
 &\stackrel{(\mref{eq:12.31})}=&\alpha_{A}(a)\triangleright \alpha_{A\otimes N}(b \otimes n).
 \end{eqnarray*}
 Similarly, $\beta_{A\otimes N}(a\triangleright(b\otimes n))=\beta_{A}(a)\triangleright \beta_{A\otimes N}(b \otimes n)$.
 \begin{eqnarray*}
 \alpha_{A}(a)\triangleright(b\triangleright (c\otimes n))
 &\stackrel{(\mref{eq:12.31})}=&\alpha_{A}(a)(bc)\otimes\beta^{2}_{N}(n)
 +\varepsilon(bc)\alpha_{A}\omega^{-1}_{A}\alpha_{A}(a)_{1}\otimes
 (\alpha_{A}(a)_{2}\triangleright\beta_{N}(n))\\
 &&+\varepsilon(c)\alpha_{A}(a)\alpha_{A}\omega^{-1}_{A}(b_{1})\otimes\beta_{N}(b_{2}\triangleright n)+\varepsilon(c)\varepsilon(b_{1})\alpha_{A}\omega^{-1}_{A}\alpha_{A}(a)_{1} \\
 &&\otimes (\alpha_{A}(a)_{2}\triangleright (b_{2}\triangleright n))\\
 &\stackrel{(\mref{eq:1.11})}=&\alpha_{A}(a)(bc)\otimes\beta^{2}_{N}(n)
 +\varepsilon(c)\alpha_{A}(a)\alpha_{A}\omega^{-1}_{A}(b_{1})\otimes\beta_{N}(b_{2}\triangleright n)\\
 &&+\varepsilon(c)\alpha^{2}_{A}\omega^{-1}_{A}(a_{1})\otimes
 (\alpha_{A}(a_{2})\triangleright(\psi_{A}(b)\triangleright n))\quad(\hbox{by Lemma \mref{lem:12.4}})\\
 &\stackrel{(\mref{eq:1.15})(\mref{eq:1.13})}=&\alpha_{A}(a)(bc)\otimes\beta^{2}_{N}(n)
 +\varepsilon(c)\alpha_{A}(a)\alpha_{A}\omega^{-1}_{A}(b_{1})\otimes(\beta_{A}(b_{2})\triangleright \beta_{N}(n))\\
 &&+\varepsilon(c)\alpha^{2}_{A}\omega^{-1}_{A}(a_{1})\otimes
 (a_{2}\psi_{A}(b)\triangleright\beta_{N}(n))\\
 &\stackrel{(\mref{eq:1.3})(\mref{eq:12.4})}=&(ab)\beta_{A}(c)\otimes\beta^{2}_{N}(n)
 +\varepsilon(c)\alpha_{A}\omega^{-1}_{A}(ab)_{1}\otimes((ab)_{2}\triangleright \beta_{N}(n))\\
 &\stackrel{(\mref{eq:12.31})}=&ab\triangleright \beta_{A\otimes N}(c\otimes n).
 \end{eqnarray*}
 Thus $(A\o V, \gamma, \alpha_{A}\o \a_V, \beta_{A}\o \b_V)$ is a left $(A,\mu,\alpha_{A},\beta_{A})$-module. Then it remains to check the compatibility condition of the left $0$-infBH-Hopf module.
 \begin{eqnarray*}
 \rho\gamma(a\otimes(b\otimes n))
 &=&\Delta(ab)\otimes\beta_{N}\psi_{N}(n)+\varepsilon(b)\Delta(\alpha_{A}\omega^{-1}_{A}(a_{1}))
 \otimes\psi_{N}(a_{2}\triangleright n)\\
 &\stackrel{(\mref{eq:12.4})}=&\omega_{A}(a)b_{1}\otimes\beta_{A}(b_{2})\otimes\beta_{N}\psi_{N}(n)
 +\alpha_{A}(a_{1})\otimes a_{2}\psi_{A}(b)\otimes \beta_{N}\psi_{N}(n)\\
 &&+\varepsilon(b)\alpha_{A}\omega^{-1}_{A}
 (a_{11})\otimes\alpha_{A}\omega^{-1}_{A}(a_{12})\otimes \psi_{N}(a_{2}\triangleright n)\\
 &\stackrel{(\mref{eq:1.9})}=&\omega_{A}(a)b_{1}\otimes\beta_{A}(b_{2})\otimes\beta_{N}\psi_{N}(n)
 +\alpha_{A}(a_{1})\otimes a_{2}\psi_{A}(b)\otimes \beta_{N}\psi_{N}(n)\\
 &&+\varepsilon(b)\alpha_{A}
 (a_{1})\otimes\alpha_{A}\omega^{-1}_{A}(a_{21})\otimes (a_{22} \triangleright \psi_{N}(n))\\
 &\stackrel{(\mref{eq:12.31})}=&(\mu\otimes \beta_{A\otimes N})(\omega_{A}(a)\otimes\Delta(b)\otimes\psi_{N}(n))+(\alpha_{A}\otimes\gamma)
 (\Delta(a)\otimes\psi_{A}(b)\otimes \psi_{N}(n))\\
 &=&(\mu\otimes \beta_{A\otimes N})(\omega_{A}\otimes\rho)(a\otimes(b\otimes n))+(\alpha_{A}\otimes\gamma)(\Delta\otimes\psi_{A\otimes N})(a\otimes(b\otimes n)).
 \end{eqnarray*}
 These complete the proof.
 \end{proof}

\end{ex}

 \begin{rmk} If $\l=0$ in Example \mref{ex:12.19} (3) and (4), then we can get Example \mref{ex:12.19} (2).
 \end{rmk}

\subsection{Modules over (anti)quasitriangular $\l$-infBH-bialgebras}\label{se:module} In this subsection, we prove that every module over (anti)quasitriangular $\l$-infBH-bialgebra can induce a  $\l$-infBH-Hopf module.

\subsubsection{First approach}
 First we provide a characterization of $\l$-infBH-bialgebra by an element $r\in A\o A$.

 Let $(A, \mu, 1, \a, \b)$ be a unitary BiHom-algebra, $\psi, \om: A\lr A$ be two linear maps such that Eqs.(\mref{eq:12.1}), (\mref{eq:12.3}) and (\mref{eq:12.30}) hold, $r\in A\o A$ be an  $\a, \b, \psi, \om$-invariant element. We mean
 an element $r\in A\o A$ is {\bf $f$-invariant} if $(f\o f)(r)=r$, where $f: A\lr A$ is a linear map. Define a linear map $\D_r: A\lr A\o A$ by
 \begin{eqnarray}
 &\Delta_{r}(a)=\alpha^{-1}(a)\triangleright r-r\triangleleft \beta^{-1}(a)-\lambda (\omega(a)\otimes 1),~~~\forall~~a\in A,&\mlabel{eq:14.5}
 \end{eqnarray}
 i.e.,
 \begin{eqnarray*}
 \Delta_{r}(a)=\omega\alpha^{-1}(a) r^{1}\otimes\beta(r^{2})-\alpha(r^{1})\otimes r^{2} \psi\beta^{-1}(a)-\lambda (\omega(a)\otimes 1).
 \end{eqnarray*}

 \begin{lem}\mlabel{lem:14.4} The map  $\Delta_{r}$ defined by Eq.(\mref{eq:14.5}) is a $\lambda$-BiHom-derivation.
 \end{lem}

 \begin{proof} We only check Eq.(\mref{eq:12.10}) for $\D_r$ as follows. For all  $a,b \in A$,
 \begin{eqnarray*}
 &&\hspace{-25mm}a\triangleright\Delta_{r}(b)+\Delta_{r}(a)\triangleleft b+\lambda\alpha\omega(a)\otimes\beta\psi(b)\\
 &\stackrel{(\mref{eq:14.1})(\mref{eq:14.2})(\mref{eq:1.5})}=&\omega(a)(\omega\alpha^{-1}(b) r^{1})\otimes\beta^{2}(r^{2})
 -\omega(a)\alpha(r^{1})\otimes \beta(r^{2})\psi(b)
 -\lambda\omega(a)\omega(b)\otimes 1\\
 &&+\omega(a)\alpha(r^{1})\otimes\beta(r^{2})\psi(b)
 -\alpha^{2}(r^{1})\otimes (r^{2} \psi\beta^{-1}(a))\psi(b)
 -\lambda\alpha\omega(a)\otimes\beta\psi(b)\\
 &&+\lambda\alpha\omega(a)\otimes\beta\psi(b)\\
 &\stackrel{(\mref{eq:1.3})}=&(\omega\alpha^{-1}(a)\omega\alpha^{-1}(b)) r^{1}\otimes\beta(r^{2})
 -\omega(a)\alpha(r^{1})\otimes \beta(r^{2})\psi(b)
 -\lambda\omega(a)\omega(b)\otimes 1\\
 &&+\omega(a)\alpha(r^{1})\otimes\beta(r^{2})\psi(b)
 -\alpha(r^{1})\otimes r^{2} (\psi\beta^{-1}(a)\psi\beta^{-1}(b))\\
 &\stackrel{(\mref{eq:12.3})(\mref{eq:1.2})}=&\omega\alpha^{-1}(ab) r^{1}\otimes\beta(r^{2})
 -\alpha(r^{1})\otimes r^{2}\psi\beta^{-1}(ab)
 -\lambda\omega(ab)\otimes 1\\
 &=&\Delta_{r}(ab),
 \end{eqnarray*}
 as desired.
 \end{proof}

 For convenience, we follow the notations in \cite{LMMP3} or \cite{MY}. Let $(A,\mu,\alpha,\beta)$ be a unitary BiHom-associative algebra, $\psi,\omega: A\lr A$ be linear maps, $r\in A\otimes A$.  We define the following elements in $A\otimes A\otimes A$:
 \begin{eqnarray*}
 &r_{12}r_{23}=\alpha(r^{1})\otimes r^{2}\bar{r}^{1}\otimes\beta(\bar{r}^{2}),~~ r_{13}r_{12}=\omega(r^{1})\bar{r}^{1}\otimes \beta(\bar{r}^{2})\otimes\alpha\psi(r^{2}),&\\
 &r_{23}r_{13}=\beta\omega(r^{1})\otimes\alpha(\bar{r}^{1})\otimes \bar{r}^{2}\psi(r^{2}),\qquad r_{13}=\omega(r^{1})\otimes 1\otimes\psi(r^{2}),&\\
 &r_{12}=r\o 1,\qquad r_{23}=1\o r.&
 \end{eqnarray*}

 \begin{thm}\mlabel{thm:14.5} Let $(A,\mu,\alpha,\beta)$ be a unitary BiHom-associative algebra such that $\alpha,\beta$ are bijective, $\psi,\omega:A\longrightarrow A$ be linear maps, $r=r^{1}\otimes r^{2}\in A\otimes A$ be $\alpha,\beta,\psi,\omega$-invariant and moreover Eqs.(\mref{eq:12.1}),  (\mref{eq:12.3}) and (\mref{eq:12.30}) hold. Then the $\l$-BiHom-derivation $\D_r$ defined by Eq.(\mref{eq:14.5}) is BiHom-coassociative if and only if
 \begin{eqnarray}
 &\omega\alpha^{-1}(a)\triangleright (r_{13}r_{12}-r_{12}r_{23}+r_{23}r_{13}-\lambda r_{13})=(r_{13}r_{12}-r_{12}r_{23}+r_{23}r_{13}-\lambda r_{13})\triangleleft\psi\beta^{-1}(a).&\label{eq:coboundary}
 \end{eqnarray}
 \end{thm}

 \begin{proof} For all $a\in A$ and $\bar{r}=r$, on one hand,
 \begin{eqnarray*}
 &&(\Delta_{r}\otimes\psi)\circ\Delta_{r}(a)\\
 &&\qquad \stackrel{(\mref{eq:12.30})}=(\omega^{2}\alpha^{-2}(a)\omega\alpha^{-1}(r^{1}))\bar{r}^{1}\otimes\beta(\bar{r}^{2})\otimes \psi\beta(r^{2})-\alpha(\bar{r}^{1})\otimes \bar{r}^{2}(\alpha^{-1}\beta^{-1}\psi\omega(a) \psi\beta^{-1}(r^{1}))\otimes \psi\beta(r^{2})\\
 &&\qquad\qquad-\lambda\omega^{2}\alpha^{-1}(a)\omega(r^{1})\otimes 1 \otimes \psi\beta(r^{2})
 -\omega(r^{1})\bar{r}^{1}\otimes\beta(\bar{r}^{2})\otimes \psi(r^{2})\psi^{2}\beta^{-1}(a)\\
 &&\qquad\qquad+\alpha(\bar{r}^{1})\otimes \bar{r}^{2}\alpha\beta^{-1}\psi(r^{1})\otimes \psi(r^{2}) \psi^{2}\beta^{-1}(a)
 +\lambda\alpha\omega(r^{1})\otimes 1\otimes \psi(r^{2}) \psi^{2}\beta^{-1}(a)\\
 &&\qquad\qquad-\lambda\alpha^{-1}\omega^{2}(a)r^{1}\otimes\beta(r^{2})\otimes 1
 +\lambda\alpha(r^{1})\otimes r^{2}\psi\beta^{-1}\omega(a)\otimes 1
 +\lambda^{2}\omega^{2}(a)\otimes 1 \otimes 1\\
 &&\qquad \stackrel{(\mref{eq:1.3})}=\omega^{2}\alpha^{-1}(a)(\omega\alpha^{-1}(r^{1})\beta^{-1}(\bar{r}^{1}))
 \otimes\beta(\bar{r}^{2})\otimes \psi\beta(r^{2})\\
 &&\qquad\qquad-\alpha(\bar{r}^{1})\otimes \bar{r}^{2}(\alpha^{-1}\beta^{-1}\psi\omega(a) \psi\beta^{-1}(r^{1}))\otimes \psi\beta(r^{2})\\
 &&\qquad\qquad-\lambda\omega^{2}\alpha^{-1}(a)\omega(r^{1})\otimes 1 \otimes \psi\beta(r^{2})
 -\omega(r^{1})\bar{r}^{1}\otimes\beta(\bar{r}^{2})\otimes \psi(r^{2})\psi^{2}\beta^{-1}(a)\\
 &&\qquad\qquad+\alpha(\bar{r}^{1})\otimes \bar{r}^{2}\alpha\beta^{-1}\psi(r^{1})\otimes \psi(r^{2}) \psi^{2}\beta^{-1}(a)
 +\lambda\alpha\omega(r^{1})\otimes 1\otimes \psi(r^{2}) \psi^{2}\beta^{-1}(a)\\
 &&\qquad\qquad-\lambda\alpha^{-1}\omega^{2}(a)r^{1}\otimes\beta(r^{2})\otimes 1
 +\lambda\alpha(r^{1})\otimes r^{2}\psi\beta^{-1}\omega(a)\otimes 1
 +\lambda^{2}\omega^{2}(a)\otimes 1 \otimes 1\\
 &&\quad\stackrel{(\mref{eq:1.5})(\mref{eq:14.3})(\mref{eq:14.4})}=\omega\alpha^{-1}(a)\triangleright
 (\omega\alpha^{-1}(r^{1})\beta^{-1}(\bar{r}^{1})\otimes\bar{r}^{2}\otimes \psi(r^{2}))\\
 &&\qquad\qquad-\alpha(\bar{r}^{1})\otimes \bar{r}^{2}(\alpha^{-1}\beta^{-1}\psi\omega(a) \psi\beta^{-1}(r^{1}))\otimes \psi\beta(r^{2})\\
 &&\qquad\qquad-\lambda\omega\alpha^{-1}(a)\triangleright(\omega(r^{1})\otimes 1 \otimes \psi(r^{2}))
 -(\omega\alpha^{-1}(r^{1})\alpha^{-1}(\bar{r}^{1})\otimes\alpha^{-1}\beta(\bar{r}^{2})\otimes \psi(r^{2}))\triangleleft\psi\beta^{-1}(a)\\
 &&\qquad\qquad+(\bar{r}^{1}\otimes \alpha^{-1}(\bar{r}^{2})\beta^{-1}\psi(r^{1})\otimes \psi(r^{2})) \triangleleft\psi\beta^{-1}(a)
 +\lambda(\omega(r^{1})\otimes 1\otimes \psi(r^{2}))\triangleleft \psi\beta^{-1}(a)\\
 &&\qquad\qquad-\lambda\omega\alpha^{-1}(a)\triangleright(r^{1}\otimes r^{2}\otimes 1)
 +\lambda\alpha(r^{1})\otimes r^{2}\psi\beta^{-1}\omega(a)\otimes 1
 +\lambda^{2}\omega^{2}(a)\otimes 1 \otimes 1\\
 &&\qquad\quad=\omega\alpha^{-1}(a)\triangleright r_{13}r_{12}
 -\alpha(\bar{r}^{1})\otimes \bar{r}^{2}(\alpha^{-1}\beta^{-1}\psi\omega(a) \beta^{-1}(r^{1}))\otimes \beta(r^{2})\\
 &&\qquad\qquad-\lambda\omega\alpha^{-1}(a)\triangleright r_{13}
 -r_{13}r_{12}\triangleleft\psi\beta^{-1}(a)+r_{12}r_{23} \triangleleft\psi\beta^{-1}(a)
 +\lambda r_{13}\triangleleft \psi\beta^{-1}(a)\\
 &&\qquad\qquad-\lambda\omega\alpha^{-1}(a)\triangleright r_{12}
 +\lambda\alpha(r^{1})\otimes r^{2}\psi\beta^{-1}\omega(a)\otimes 1
 +\lambda^{2}\omega^{2}(a)\otimes 1 \otimes 1.
 \end{eqnarray*}
 On the other hand,
 \begin{eqnarray*}
 &&(\omega\otimes\Delta_{r})\circ\Delta_{r}(a)\\
 &&\quad \stackrel{(\mref{eq:12.30})}=\omega^{2}\alpha^{-1}(a) \omega(r^{1})\otimes (\omega\alpha^{-1}\beta(r^{2})\bar{r}^{1}\otimes\beta(\bar{r}^{2})-\alpha(\bar{r}^{1})\otimes \bar{r}^{2}\psi\beta^{-1}\beta(r^{2})-\lambda\omega\beta(r^{2})\otimes 1)\\
 &&\qquad-\omega\alpha(r^{1})\otimes(\omega\alpha^{-1}(r^{2}\psi\beta^{-1}(a))\bar{r}^{1}\otimes\beta(\bar{r}^{2})
 -\alpha(\bar{r}^{1})\otimes \bar{r}^{2}\psi\beta^{-1}(r^{2}\psi\beta^{-1}(a))\\
 &&\qquad-\lambda\omega(r^{2}\psi\beta^{-1}(a))\otimes 1)-\lambda \omega^{2}(a)\otimes (1\cdot r^{1}\otimes \beta(r^{2})-\alpha(r^{1})\otimes r^{2}\cdot 1-\lambda 1\otimes 1)\\
 &&\quad \stackrel{(\mref{eq:1.5})}=\omega^{2}\alpha^{-1}(a)\omega(r^{1})\otimes \omega\alpha^{-1}\beta(r^{2})\bar{r}^{1}\otimes\beta(\bar{r}^{2})
 -\omega^{2}\alpha^{-1}(a)\omega(r^{1})\otimes\alpha(\bar{r}^{1})\otimes \bar{r}^{2}\psi(r^{2})\\
 &&\qquad-\lambda\omega^{2}\alpha^{-1}(a) \omega(r^{1})\otimes\omega\beta(r^{2})\otimes 1
 -\omega\alpha(r^{1})\otimes(\omega\alpha^{-1}(r^{2})
 \alpha^{-1}\beta^{-1}\psi\omega(a))\bar{r}^{1}\otimes\beta(\bar{r}^{2})\\
 &&\qquad+\omega\alpha(r^{1})\otimes\alpha(\bar{r}^{1})\otimes \bar{r}^{2}(\psi\beta^{-1}(r^{2})\psi^{2}\beta^{-2}(a))
 +\lambda\omega\alpha(r^{1})\otimes\omega(r^{2})\omega\psi\beta^{-1}(a)\otimes 1\\
 &&\qquad-\lambda\omega^{2}(a)\otimes r^{1}\otimes r^{2}
 +\lambda \omega^{2}(a)\otimes r^{1}\otimes r^{2}
 +\lambda^{2} \omega^{2}(a)\otimes  1\otimes 1\\
 &&\quad\stackrel{(\mref{eq:1.3})}=\omega^{2}\alpha^{-1}(a)\omega(r^{1})\otimes \omega\alpha^{-1}\beta(r^{2})\bar{r}^{1}\otimes\beta(\bar{r}^{2})
 -\omega^{2}\alpha^{-1}(a)\omega(r^{1})\otimes\alpha(\bar{r}^{1})\otimes \bar{r}^{2}\psi(r^{2})\\
 &&\qquad-\lambda\omega^{2}\alpha^{-1}(a) \omega(r^{1})\otimes\omega\beta(r^{2})\otimes 1
 -\omega\alpha(r^{1})\otimes \omega(r^{2})
 (\alpha^{-1}\beta^{-1}\psi\omega(a)\beta^{-1}(\bar{r}^{1}))\otimes\beta(\bar{r}^{2})\\
 &&\qquad+\omega\alpha(r^{1})\otimes\alpha(\bar{r}^{1})\otimes (\alpha^{-1}(\bar{r}^{2})\psi\beta^{-1}(r^{2}))\psi^{2}\beta^{-1}(a))
 +\lambda\omega\alpha(r^{1})\otimes\omega(r^{2})\omega\psi\beta^{-1}(a)\otimes 1\\
 &&\qquad+\lambda^{2} \omega^{2}(a)\otimes  1\otimes 1\\
 &&~\stackrel{(\mref{eq:14.3})(\mref{eq:14.4})}=\omega\alpha^{-1}(a)\triangleright(\omega(r^{1})\otimes \omega\alpha^{-1}(r^{2})\beta^{-1}(\bar{r}^{1})\otimes \bar{r}^{2})\\
 &&\qquad-\omega\alpha^{-1}(a)\triangleright(\omega(r^{1})\otimes\alpha\beta^{-1}(\bar{r}^{1})\otimes \beta^{-1}(\bar{r}^{2})\psi\beta^{-1}(r^{2}))\\
 &&\qquad-\lambda\omega\alpha^{-1}(a)\triangleright(\omega(r^{1})\otimes\omega(r^{2})\otimes 1)
 -\omega\alpha(r^{1})\otimes \omega(r^{2})
 (\alpha^{-1}\beta^{-1}\psi\omega(a)\beta^{-1}(\bar{r}^{1}))\otimes\beta(\bar{r}^{2})\\
 &&\qquad+(\omega(r^{1})\otimes \bar{r}^{1}\otimes \alpha^{-1}(\bar{r}^{2})\psi\beta^{-1}(r^{2}))\triangleleft\psi\beta^{-1}(a))
 +\lambda\omega\alpha(r^{1})\otimes\omega(r^{2})\omega\psi\beta^{-1}(a)\otimes 1\\
 &&\qquad+\lambda^{2} \omega^{2}(a)\otimes 1\otimes 1\\
 &&\quad =\omega\alpha^{-1}(a)\triangleright r_{12}r_{23}
 -\omega\alpha^{-1}(a)\triangleright r_{23}r_{13}
 -\lambda\omega\alpha^{-1}(a)\triangleright r_{12}\\
 &&\qquad-\alpha(r^{1})\otimes r^{2}(\alpha^{-1}\beta^{-1}\psi\omega(a) \beta^{-1}(\bar{r}^{1}))\otimes \beta(\bar{r}^{2})
 +r_{23}r_{13}\triangleleft\psi\beta^{-1}(a)\\
 &&\qquad+\lambda\alpha(r^{1})\otimes r^{2}\omega\psi\beta^{-1}(a)\otimes 1
 +\lambda^{2} \omega^{2}(a)\otimes 1\otimes 1.
 \end{eqnarray*}
 Therefore $\Delta_{r}$ is BiHom-coassociative if and only if Eq.(\mref{eq:coboundary}) holds, finishing the proof.
 \end{proof}

 Based on the above result, we introduce the notion of nonhomogeneous associative BiHom-Yang-Baxter equation.

 \begin{defi}\label{de:waybe} Let $(A,\mu,1,\alpha,\beta)$ be a unitary BiHom-associative algebra, $\psi,\omega: A\lr A$ be linear maps and $r\in A\otimes A$. We call
 \begin{eqnarray}
 r_{13}r_{12}-r_{12}r_{23}+r_{23}r_{13}=\lambda r_{13} \label{eq:waybe}
 \end{eqnarray}
 the {\bf $\l$-associative BiHom-Yang-Baxter equation (abbr. $\l$-abhYBe) in $(A,\mu,1,\alpha,\beta)$} where $\l$ is a given element in $K$.
 \end{defi}

 By Lemma \mref{lem:14.4} and Theorem \mref{thm:14.5}, we have

 \begin{cor}\mlabel{cor:waybe} Let $(A,\mu,1,\alpha,\beta)$ be a unitary BiHom-associative algebra such that $\alpha,\beta$ are bijective, $\psi,\omega: A\longrightarrow A$ be linear maps, $r=r^{1}\otimes r^{2}\in A\otimes A$ be $\alpha,\beta,\psi,\omega$-invariant and moreover Eqs.(\mref{eq:12.1}),  (\mref{eq:12.3}) and (\mref{eq:12.30}) hold. If $r$ is a solution of the $\l$-abhYBe, then $(A,\mu, \D_r, 1,\alpha,\beta,\psi,\om)$ is a $\l$-infBH-bialgebra, where $\D_r$ is defined by Eq.(\mref{eq:14.5}).
 \end{cor}

 \begin{defi}\mlabel{de:14.6} Under the assumptions of Corollary \mref{cor:waybe}, a {\bf quasitriangular unitary $\l$-infBH-bialgebra} is a 8-tuple $(A, \mu, 1, \alpha, \beta, \psi, \omega, r)$ consisting of a unitary BiHom-associative algebra $(A, \mu, 1, \alpha, \beta)$ and a solution $r\in A\otimes A$ of a $\l$-abhYBe.
 \end{defi}

 \begin{pro}\mlabel{pro:14.8} Under the assumptions of Corollary \mref{cor:waybe},  $(A, \mu, \D=\D_r, 1, \alpha, \beta, \psi, \omega)$, where $\D_r$ is defined by Eq.(\mref{eq:14.5}), is a quasitriangular unitary $\l$-infBH-bialgebra if and only if
 \begin{eqnarray}
 &(\Delta\otimes\psi)(r)=-r_{23}r_{13}&\mlabel{eq:14.8}
 \end{eqnarray}
 or
 \begin{eqnarray}
 &(\omega\otimes \Delta)(r)=r_{13}r_{12}-\lambda(r_{13}+r_{12}).&\mlabel{eq:14.9}
 \end{eqnarray}
 holds.
 \end{pro}

 \begin{proof} It is sufficient to prove that Eq.(\mref{eq:waybe}) is equivalent to Eq.(\mref{eq:14.8}) or Eq.(\mref{eq:14.9}). While
 \begin{eqnarray*}
 (\Delta\otimes\psi)(r)
 &\stackrel{(\mref{eq:14.5})}=&\omega\alpha^{-1}(r^{1})\bar{r}^{1}\otimes\beta(\bar{r}^{2})\otimes \psi(r^{2})-\alpha(\bar{r}^{1})\otimes \bar{r}^{2}\psi\beta^{-1}(r^{1})\otimes \psi(r^{2})\\
 &&-\lambda\omega(r^{1})\otimes 1\otimes \psi(r^{2})\\
 &=&r_{13}r_{12}-r_{12}r_{23}-\lambda r_{13},
 \end{eqnarray*}
 and
 \begin{eqnarray*}
 (\omega\otimes\Delta)(r)&\stackrel{(\mref{eq:14.5})}=&\omega(r^{1})\otimes \omega\alpha^{-1}(r^{2})\bar{r}^{1}\otimes \beta(\bar{r}^{2})-\omega(r^{1})\otimes \alpha(\bar{r}^{1})\otimes \bar{r}^{2}\psi\beta^{-1}(r^{2})\\
 &&-\lambda\omega(r^{1})\otimes \omega(r^{2})\otimes 1\\
 &=&r_{12}r_{23}-r_{23}r_{13}-\lambda r_{12},
 \end{eqnarray*}
 as desired.
 \end{proof}

 \begin{rmk} If $\a=\b=\psi=\om=\id$ and $\l=0$ in Proposition \mref{pro:14.8}, then we can obtain \cite[Proposition 5.5]{Ag99}. We notice  here that any one of Eqs.(\mref{eq:14.8}) and (\mref{eq:14.9}) is equivalent to Eq.(\mref{eq:waybe}).
 \end{rmk}

 $\lambda$-infBH-Hopf modules can be obtained from the modules over quasitriangular unitary $\lambda$-infBH-bialgebra by the following procedure.

 \begin{thm}\mlabel{thm:12.02} Let $(A, \mu, 1, \alpha_{A}, \beta_{A}, \psi_{A}, \omega_{A}, r)$ be a quasitriangular unitary $\lambda$-infBH-bialgebra and $(M, \gamma, \alpha_{M}, \beta_{M})$ be a left $(A, \mu, \alpha_{A}, \beta_{A})$-module, $\psi_{M}, \omega_{M}: M\longrightarrow M$ be linear maps such that $\beta_{M}\ci \psi_{M}=\psi_{M}\ci \beta_{M}$, $\psi_{M}\ci \g=\g\ci (\psi_{A}\o \psi_{M})$.
 Then $(M, \gamma, \rho, \alpha_{M}, \beta_{M}, \psi_{M}, \omega_{M})$ becomes a $\lambda$-infBH-Hopf module with the coaction $\rho: M\longrightarrow A\otimes M$ given by
 \begin{eqnarray}\label{eq:modulefromqt}
 &\rho(m):=-\alpha_{A}(r^{1})\otimes r^{2}\triangleright\psi_{M}\beta_{M}^{-1}(m), \forall~ m\in M.&
 \end{eqnarray}
 \end{thm}

 \begin{proof} We first prove that $(M, \rho, \psi_{M}, \omega_{M})$ is a left $(A, \D_r, \psi_{A}, \omega_{A})$-comodule as follows. For all $m\in M$, we have
 \begin{eqnarray*}
 (\Delta_r\otimes \psi_{M})\rho(m)
 &\stackrel{(\mref{eq:14.8})}=&\beta_{A}\omega_{A}(\bar{r}^{1})\otimes \alpha_{A}(r^{1})\otimes \alpha_{A}^{-1}(r^{2})\alpha_{A}^{-1}\psi_{A}(\bar{r}^{2})\triangleright\psi_{M}^{2}\beta_{M}^{-1}(m)\\
 &\stackrel{(\mref{eq:1.15})}=&\alpha_{A}\omega_{A}(\bar{r}^{1})\otimes \alpha_{A}(r^{1})\otimes r^{2}\triangleright(\psi_{A}\beta_{A}^{-1}(\bar{r}^{2})\triangleright\psi_{M}^{2}\beta_{M}^{-2}(m))\\
 &=&(\omega_{A}\otimes \rho)\rho(m).
 \end{eqnarray*}
 Next we then check the compatibility condition. For all $a\in A$ and $m\in M$, we have
 \begin{eqnarray*}
 &&\omega_{A}(a)m_{-1}\otimes\beta_{M}(m_{0})+\alpha_{A}(a_{1})\otimes a_{2}\triangleright\psi_{M}(m)+\lambda\alpha_{A}\omega_{A}(a)\otimes\beta_{M}\psi_{M}(m)\\
 &&\qquad\stackrel{(\mref{eq:14.5})}=-\omega_{A}(a)\alpha_{A}(r^{1})\otimes\beta_{A}(r^{2})\triangleright\psi_{M}(m)
 +\alpha_{A}(\omega_{A}\alpha_{A}^{-1}(a)r^{1})\otimes \beta_{A}(r^{2})\triangleright\psi_{M}(m)\\
 &&\qquad\quad\qquad-\alpha_{A}^{2}(r^{1})\otimes r^{2}\psi_{A}\beta_{A}^{-1}(a)\triangleright\psi_{M}(m)
 -\lambda\alpha_{A}\omega_{A}(a)\otimes1\triangleright\psi_{M}(m)\\
 &&\qquad\quad\qquad+\lambda\alpha_{A}\omega_{A}(a)\otimes\beta_{M}\psi_{M}(m)\\
 &&\qquad\stackrel{(\mref{eq:1.2})(\mref{eq:1.5})}=-\alpha_{A}(r^{1})\otimes \alpha_{A}^{-1}(r^{2})\psi_{A}\beta_{A}^{-1}(a)\triangleright\psi_{M}(m)\\
 &&\hspace{2mm}\qquad\stackrel{(\mref{eq:1.15})}=-\alpha_{A}(r^{1})\otimes r^{2}(\psi_{A}\beta_{A}^{-1}(a)\triangleright\psi_{M}\beta_{M}^{-1}(m))\\
 &&\qquad\quad =(a\triangleright m)_{-1}\otimes (a\triangleright m)_{0},
 \end{eqnarray*}
 completing the proof.
 \end{proof}

\subsubsection{Second approach}
 Inspired by Example \mref{ex:12.3}, we can define a new comultiplication for $\l$-infBH-bialgebra by replacing $\l(\om(a)\o 1)$ in Eq.(\mref{eq:14.5}) by $\lambda (1\otimes \psi(a))$, i.e.,
 \begin{eqnarray}
 &\widetilde{\Delta}_{r}(a)=\alpha^{-1}(a)\triangleright r-r\triangleleft \beta^{-1}(a)-\lambda (1\otimes \psi(a)).&\mlabel{eq:14.25}
 \end{eqnarray}
 In this case, we only list the parallel results and omit the partial proofs.

 \begin{pro}\mlabel{pro:14.21} The $\widetilde{\Delta}_{r}$ defined by Eq.(\mref{eq:14.25}) is a $\lambda$-BiHom-derivation.
 \end{pro}

 \begin{thm}\mlabel{thm:14.22}Let $(A,\mu,\alpha,\beta)$ be a unitary BiHom-associative algebra such that $\alpha,\beta$ are bijective, $\psi,\omega:A\longrightarrow A$ be linear maps, $r=r^{1}\otimes r^{2}\in A\otimes A$ be $\alpha,\beta,\psi,\omega$-invariant and moreover Eqs.(\mref{eq:12.1}),  (\mref{eq:12.3}) and (\mref{eq:12.30}) hold. Then the $\l$-BiHom-derivation $\widetilde{\D}_r$ defined by Eq.(\mref{eq:14.25}) is BiHom-coassociative if and only if
 \begin{eqnarray}
 &\omega\alpha^{-1}(a)\triangleright (r_{13}r_{12}-r_{12}r_{23}+r_{23}r_{13}+\lambda r_{13})=(r_{13}r_{12}-r_{12}r_{23}+r_{23}r_{13}+\lambda r_{13})\triangleleft\psi\beta^{-1}(a).&\label{eq:coboundary1}
 \end{eqnarray}
 \end{thm}

 \begin{rmk}
 If we substitute $\l$ in Eq.(\mref{eq:coboundary}) by  $-\l$, then we  obtain Eq.(\mref{eq:coboundary1}).
 \end{rmk}

 \begin{defi}\mlabel{de:14.6a} Under the assumptions of Corollary \mref{cor:waybe}, an {\bf anti-quasitriangular unitary $\l$-infBH-bialgebra} is a 8-tuple $(A, \mu, 1, \alpha, \beta, \psi, \omega, r)$ consisting of a unitary BiHom-associative algebra $(A, \mu, 1, \alpha, \beta)$ and a solution $r\in A\otimes A$ of a $(-\l)$-abhYBe.
 \end{defi}

 \begin{pro}\mlabel{pro:14.25} Under the assumption of Corollary \mref{cor:waybe},  $(A, \mu, \D=\widetilde{\D}_r, 1, \alpha, \beta, \psi, \omega)$, where $\widetilde{\D}_r$ is defined by Eq.(\mref{eq:14.25}), is an anti-quasitriangular unitary $\l$-infBH-bialgebra if and only if
 \begin{eqnarray}
 &(\Delta\otimes\psi)(r)=-r_{23}r_{13}-\lambda(r_{23}+r_{13})&\mlabel{eq:14.28}
 \end{eqnarray}
 or
 \begin{eqnarray}
 &(\omega\otimes \Delta)(r)=r_{13}r_{12}.&\mlabel{eq:14.29}
 \end{eqnarray}
 \end{pro}

 \begin{proof} The result can be proved by the following equalities:
 \begin{eqnarray*}
 (\Delta\otimes\psi)(r)&\stackrel{(\mref{eq:14.25})}=&\omega\alpha^{-1}(r^{1})\bar{r}^{1}\otimes\beta(\bar{r}^{2})\otimes \psi(r^{2})-\alpha(\bar{r}^{1})\otimes \bar{r}^{2}\psi\beta^{-1}(r^{1})\otimes \psi(r^{2})\\
 &&-\lambda(1\otimes \psi(r^{1}))\otimes \psi(r^{2})\\
 &=&r_{13}r_{12}-r_{12}r_{23}-\lambda r_{23},
 \end{eqnarray*}
 and
 \begin{eqnarray*}
 (\omega\otimes\Delta)(r)
 &\stackrel{(\mref{eq:14.25})}=&\omega(r^{1})\otimes \omega\alpha^{-1}(r^{2})\bar{r}^{1}\otimes \beta(\bar{r}^{2})-\omega(r^{1})\otimes \alpha(\bar{r}^{1})\otimes \bar{r}^{2}\psi\beta^{-1}(r^{2})\\
 &&-\lambda\omega(r^{1})\otimes 1\otimes \psi(r^{2})\\
 &=&r_{12}r_{23}-r_{23}r_{13}-\lambda r_{13},
 \end{eqnarray*}
 as desired.
 \end{proof}

 \begin{rmk}\label{rmk:qcq} (1) If $\l=0$, then the conditions in Proposition \mref{pro:14.8} and the ones in Proposition \mref{pro:14.25} are consistent, and in this case Proposition \mref{pro:14.8} and \mref{pro:14.25} are the BiHom-version of \cite[Proposition 5.5]{Ag99}.

 (2) If $\l\neq 0$, then the conditions in Proposition \mref{pro:14.8} and the ones in Proposition \mref{pro:14.25} are different which shows that there are essential differences between quasitriangular unitary $\l$-infBH-bialgebras and the corresponding anti-case.

 (3) When $\l=-1$ in Proposition \mref{pro:14.25}, we obtain \cite[Corollary 2.32]{MLi}. If further, the structure maps $\a=\b=\psi=\om=\id$, Proposition \mref{pro:14.25} is consistent with \cite[Lemma 3.19]{Br}.
 \end{rmk}

 $\lambda$-infBH-Hopf modules also can be constructed from the modules over anti-quasitriangular unitary $\lambda$-infBH-bialgebras, which is different from Theorem \mref{thm:12.02} for the quasitriangular case.

 \begin{thm}\mlabel{thm:12.02a} Let $(A, \mu, 1, \alpha_{A}, \beta_{A}, \psi_{A}, \omega_{A}, r)$ be an anti-quasitriangular unitary $\lambda$-infBH-bialgebra and $(M, \widetilde{\gamma}, \alpha_{M}, \beta_{M})$ a left $(A, \mu, \alpha_{A}, \beta_{A})$-module, $\psi_{M}, \omega_{M}: M\longrightarrow M$ be linear maps  such that $\beta_{M}\ci \psi_{M}=\psi_{M}\ci \beta_{M}$, $\psi_{M}\ci \widetilde{\gamma}=\widetilde{\gamma}\ci (\psi_{A}\o \psi_{M})$. Then $(M, \widetilde{\gamma}, \widetilde{\rho}, \alpha_{M}, \beta_{M}, \psi_{M}, \omega_{M})$ becomes a left $\lambda$-infBH-Hopf module with the coaction $\widetilde{\rho}: M\longrightarrow A\otimes M$ given by
 \begin{eqnarray}\label{eq:modulefromantiqt}
 &\widetilde{\rho}(m):=-\alpha_{A}(r^{1})\otimes r^{2}\triangleright\psi_{M}\beta_{M}^{-1}(m)-\lambda 1\otimes \psi_{M}(m), \forall~ m\in M.&
 \end{eqnarray}
 \end{thm}

 \begin{proof} We first prove that $(M, \widetilde{\rho}, \psi_{M}, \omega_{M})$ is a left $(A, \widetilde{\D_r}, \psi_{A}, \omega_{A})$-comodule as follows. For all $m\in M$, we have
 \begin{eqnarray*}
 (\widetilde{\Delta_r}\otimes \psi_{M})\widetilde{\rho}(m)
 \hspace{-3mm}&\stackrel{(\mref{eq:14.28})}=&\hspace{-3mm}
 \lambda\omega_{A}(r^{1})\otimes1\otimes\alpha_{A}^{-1}\psi_{A}(r^{2})
 \triangleright\psi_{M}^{2}\beta_{M}^{-1}(m)+\lambda 1\otimes r^{1} \otimes \alpha_{A}^{-1}(r^{2})\triangleright\psi_{M}^{2}\beta_{M}^{-1}(m)\\
 &&\hspace{-3mm}+\beta_{A}\omega_{A}(r^{1})\otimes \alpha_{A}(\bar{r}^{1})\otimes \alpha_{A}^{-1}(\bar{r}^{2})\alpha_{A}^{-1}\psi_{A}(r^{2})\triangleright\psi_{M}^{2}\beta_{M}^{-1}(m) +\lambda^{2} 1\otimes 1\otimes \psi_{M}^{2}(m)\\
 \hspace{-3mm}&\stackrel{(\mref{eq:1.15})}=&\hspace{-3mm}
 \omega_{A}\alpha_{A}(r^{1})\otimes \alpha_{A}(\bar{r}^{1})\otimes \bar{r}^{2}\triangleright(\psi_{A}\beta_{A}^{-1}(r^{2})
 \triangleright\psi_{M}^{2}\beta_{M}^{-2}(m))+\lambda^{2}1\otimes 1\otimes \psi^{2}_{M}(m)\\
 &&\hspace{-3mm}+\lambda\omega_{A}\alpha_{A}(r^{1})\otimes 1\otimes\psi_{A}(r^{2})\triangleright\psi^{2}_{M}\beta^{-1}_{M}(m)
 +\lambda 1\otimes\alpha_{A}(r^{1})\otimes r^{2}\triangleright\psi^{2}_{M}\beta^{-1}_{M}(m)\\
 \hspace{-3mm}&\stackrel{(\mref{eq:12.30})}=&\hspace{-3mm}(\omega_{A}\otimes \widetilde{\rho})\widetilde{\rho}(m).
 \end{eqnarray*}
 The compatibility condition can be checked as follows. For all $a\in A$ and $m\in M$,
 \begin{eqnarray*}
 &&\hspace{-15mm}\omega_{A}(a)m_{-1}\otimes\beta_{M}(m_{0})+\alpha_{A}(a_{1})\otimes a_{2}\triangleright\psi_{M}(m)+\lambda\alpha_{A}\omega_{A}(a)\otimes\beta_{M}\psi_{M}(m)\\
 \hspace{-5mm}&\stackrel{(\mref{eq:14.25})}=&\hspace{-3mm}
 -\omega_{A}(a)\alpha_{A}(r^{1})\otimes\beta_{A}(r^{2})\triangleright\psi_{M}(m)
 +\lambda\omega_{A}(a)\cdot 1\otimes\beta_{M}\psi_{M}(m)+\omega_{A}(a)\alpha_{A}(r^{1})\otimes \beta_{A}(r^{2})\triangleright\psi_{M}(m)\\
 \hspace{-5mm}&&\hspace{-3mm}-\alpha_{A}^{2}(r^{1})\otimes r^{2}\psi_{A}\beta_{A}^{-1}(a)\triangleright\psi_{M}(m)
 -\lambda1\otimes\psi_{A}(a)\triangleright\psi_{M}(m)
 +\lambda\alpha_{A}\omega_{A}(a)\otimes\beta_{M}\psi_{M}(m)\\
 \hspace{-5mm}&\stackrel{(\mref{eq:1.5})}=&\hspace{-3mm}-\alpha_{A}(r^{1})\otimes \alpha^{-1}_{A}(r^{2})\psi_{A}\beta_{A}^{-1}(a)\triangleright\psi_{M}(m)
 -\lambda1\otimes\psi_{A}(a)\triangleright\psi_{M}(m)\\
 \hspace{-5mm}&\stackrel{(\mref{eq:1.13})}=&\hspace{-3mm}-\alpha_{A}(r^{1})\otimes r^{2}\triangleright\psi_{M}\beta^{-1}_{M}(a \triangleright m)
 -\lambda1\otimes\psi_{M}(a \triangleright m)=(a\triangleright m)_{-1}\otimes (a\triangleright m)_{0},
 \end{eqnarray*}
 as desired.
 \end{proof}

\subsection{BiHom-pre-Lie algebras from $\l$-infBH-bialgebras}\label{se:subprelie}
 In this subsection, we provide two approaches to construct BiHom-pre-Lie algebras from $\l$-infBH-bialgebras.

 Rota-Baxter operators on a BiHom-associative algebras (here we call this structure Rota-Baxter BiHom-associative algebra) were studied in \cite[Definition 2.1]{LMMP1} or \cite[Definition 2.1]{MYZZ}.

 \begin{lem}\mlabel{lem:014.10}(\cite[Corollary 2.6]{MLi}) Let $(A, \mu, 1, \alpha, \beta)$ be a unitary BiHom-associative algebra, $\psi,\omega: A\lr A$ be linear maps such that $\alpha,\beta,\psi,\omega$ are bijective and Eqs.(\mref{eq:12.1}), (\mref{eq:12.3}) hold. Assume that $r$ is a solution of  $(\pm\l)$-abhYBe in $(A,\mu,\alpha,\beta)$ and $r$ is $\alpha,\beta,\psi,\omega$-invariant. Define $R: A\lr A$ by
 \begin{eqnarray*}
 &R(a)=\mp\beta^{2}\psi(r^{1})(\alpha^{-1}\beta^{-1}(a)\alpha\omega(r^{2})), \forall a\in A.&
 \end{eqnarray*}
 Then $(A, \mu, R, \alpha, \beta)$ is a Rota-Baxter BiHom-associative algebra of weight $\l$.
 \end{lem}

 The BiHom-version of Loday's dendriform algebra was introduced in \cite[Definition 3.1]{LMMP1}.

 \begin{lem}\mlabel{lem:14.11} 
 Let $(A,R,\alpha,\beta)$ be a Rota-Baxter BiHom-associative algebra of weight $\lambda$ and  $\prec,\succ: A\otimes A\lr A$ be linear maps defined by
 \begin{eqnarray*}
 &a\prec b=a R(b)+\lambda a b,\quad a\succ b=R(a)b&
 \end{eqnarray*}
 (resp.
 \begin{eqnarray*}
 &a\prec b=a R(b),\quad a\succ b=R(a)b+\lambda a b&)
 \end{eqnarray*}
 for all $a,b\in A$. Then $(A, \prec, \succ, \alpha, \beta)$ is a BiHom-dendriform algebra.
 \end{lem}

 \begin{proof} It can be derived by \cite[Lemma 2.3, Theorem 2.5]{MYZZ}.
 \end{proof}

 \begin{rmk} When $\l=0$, Lemma \mref{lem:14.11} recovers \cite[Corollary 4.4]{LMMP1}.
 \end{rmk}

 \begin{cor}\mlabel{cor:14.12} Let $(A,\mu,1,\alpha,\beta,\psi,\omega,r)$ be a quasitriangular unitary $\l$-infBH-bialgebra. Define two binary operations $\prec,\succ$ on $A$ by
 \begin{eqnarray*}
 &a\succ b=-(\beta^{2}\psi(r^{1})(\alpha^{-1}\beta^{-1}(a)\alpha\omega(r^{2})))b,\ a\prec b=-a(\beta^{2}\psi(r^{1})(\alpha^{-1}\beta^{-1}(b)\alpha\omega(r^{2})))+\lambda ab&
 \end{eqnarray*}
 or
 \begin{eqnarray*}
 &a\succ b=-(\beta^{2}\psi(r^{1})(\alpha^{-1}\beta^{-1}(a)\alpha\omega(r^{2})))b+\lambda ab,\ a\prec b=-a(\beta^{2}\psi(r^{1})(\alpha^{-1}\beta^{-1}(b)\alpha\omega(r^{2}))).&
 \end{eqnarray*}
 Then the 5-tuple $(A,\prec,\succ,\alpha,\beta)$ is a BiHom-dendriform algebra.
 \end{cor}

 \begin{proof}
 The result follows from Lemma \mref{lem:014.10} and Lemma \mref{lem:14.11}.
 \end{proof}

 \begin{rmk} If we delete the minus signs ``$-$" in Corollary \mref{cor:14.12}, then the corresponding results for the case of anti-quasitriangular unitary $\l$-infBH-bialgebra can be obtained.
 \end{rmk}

 \begin{defi}\mlabel{de:13.1} (\cite[Definition 3.1]{LMMP6}) A {\bf (left) BiHom-pre-Lie algebra} $(A,\cdot,\alpha,\beta)$ is a 4-tuple in which $A$ is a vector space and $\cdot: A \otimes A \lr  A$, $\alpha,\beta: A \lr  A$ are linear maps satisfying $\alpha\circ\beta=\beta\circ\alpha$, $\alpha(a\cdot b)=\alpha(a)\cdot\alpha(b)$, $\beta(a\cdot b)=\beta(a)\cdot\beta(b)$ and
 \begin{eqnarray}
 &\alpha\beta(a)\cdot(\alpha(b)\cdot c)-(\beta(a)\cdot\alpha(b))\cdot\beta(c)=\alpha\beta(b)\cdot(\alpha(a)\cdot c)
 -(\beta(b)\cdot\alpha(a))\cdot\beta(c),&\mlabel{eq:13.1}
 \end{eqnarray}
 for all $a,b,c \in A$.
 \end{defi}

 Now we get a new construction of BiHom-pre-Lie algebra which ensures that the following diagram is commutative.
$$\hspace{10mm} \xymatrix@C4.3em{
&\text{Quasitriangular}\atop \text{$\l$-infBH Bialgebras}
\ar@2{->}^{\text{Lemma}\ \mref{lem:014.10}}[r]
\ar@2{->}_{\text{Corollary}\  \mref{cor:waybe}}[d]
&\text{$\lambda$-Rota Baxter}\atop \text{BiHom-algebras}
\ar@2{->}^{\text{Lemma}\ \mref{lem:14.11}}[r]
&\text{BiHom-Dendriform}\atop \text{algebras}\\
&\text{$\lambda$-infBH}\atop \text{Bialgebras}
\ar@2{->}^{\text{Theorem}\  \mref{thm:13.3}}[rr]
&&\text{BiHom-Pre-Lie}\atop \text{algebras}
\ar@2{<-}^{\text{\cite[Proposition 3.6]{LMMP6}}}[u]
&&}
$$

 \begin{thm}\mlabel{thm:13.3} Let $(A,\mu,\Delta,\alpha,\beta,\psi,\omega)$ be a $\l$-infBH-bialgebra such that $\alpha,\beta,\psi,\omega$ are invertible. Then $(A,\star,\alpha,\beta)$ is a BiHom-pre-Lie algebra, where
 \begin{eqnarray}
 &\star:A\otimes A\longrightarrow A,\ a\star b=(\alpha^{-2}\beta\omega^{-1}(b_{1})\beta^{-1}(a))\psi^{-1}(b_{2}).&\mlabel{eq:13.2}
 \end{eqnarray}
 \end{thm}

 \begin{proof} For all $a,b,c \in A$, we calculate
 \begin{eqnarray*}
 \Delta(\alpha(b)\star c)\hspace{-4mm}&\stackrel{(\mref{eq:1.7})(\mref{eq:12.3})(\mref{eq:12.4})(\mref{eq:1.2})}=&
 \hspace{-4mm}(\alpha^{-2}\beta(c_{1})
 \alpha\beta^{-1}\omega(b))\psi^{-1}(c_{21})\otimes \beta\psi^{-1}(c_{22})
 +\alpha^{-1}\beta(c_{1})\alpha^{2}\beta^{-1}(b_{1})\otimes \alpha(b_{2})c_{2}\\
 &&+\beta\omega^{-1}(c_{11})\otimes (\alpha^{-2}\beta\omega^{-1}(c_{12})\alpha\beta^{-1}\psi(b))c_{2}\\
 &&+\lambda\beta(c_{1})\otimes\alpha\psi(b)c_{2}
 +\lambda\alpha^{-1}\beta(c_{1})\alpha^{2}\beta^{-1}\omega(b)\otimes\beta(c_{2})
 \end{eqnarray*}
 and
 \begin{eqnarray*}
 \alpha\beta(a)\star(\alpha(b)\star c)&=&
 (((\alpha^{-4}\beta^{2}\omega^{-1}(c_{1})
 \alpha^{-1}(b))\alpha^{-2}\beta\psi^{-1}\omega^{-1}(c_{21}))\alpha(a))\beta\psi^{-2}(c_{22})\\
 &&+((\alpha^{-3}\beta^{2}\omega^{-1}(c_{1})
 \omega^{-1}(b_{1}))\alpha(a))(\alpha\psi^{-1}(b_{2})\psi^{-1}(c_{2}))\\
 &&+(\alpha^{-2}\beta^{2}\omega^{-2}(c_{11})\alpha(a)) ((\alpha^{-2}\beta\psi^{-1} \omega^{-1}(c_{12})\alpha\beta^{-1}(b))\psi^{-1} (c_{2}))\\
 &&+\lambda(\alpha^{-2}\beta^{2}\omega^{-1}(c_{1})\alpha(a))(\alpha(b)\psi^{-1}(c_{2})) +\lambda((\alpha^{-3}\beta^{2}\omega^{-1}(c_{1})b)
 \alpha(a))\psi^{-1}\beta(c_{2}).
 \end{eqnarray*}
 Morever,
 \begin{eqnarray*}
 (\beta(a)\star\alpha(b))\star\beta(c)
 &\stackrel{(\mref{eq:12.2})}=&(\alpha^{-2}\beta^{2}\omega^{-1}(c_{1})
 ((\alpha^{-1}\omega^{-1}(b_{1})\beta^{-1}(a))\alpha\beta^{-1}\psi^{-1}(b_{2})))\psi^{-1}\beta(c_{2})\\
 &\stackrel{(\mref{eq:1.3})}=& ((\alpha^{-3}\beta^{2}\omega^{-1}(c_{1})
 \omega^{-1}(b_{1}))\alpha(a))(\alpha\psi^{-1}(b_{2})\psi^{-1}(c_{2})).
 \end{eqnarray*}
 Hence,
 \begin{eqnarray*}
 &&\hspace{-20mm}\alpha\beta(a)\star(\alpha(b)\star c)-(\beta(a)\star\alpha(b))\star\beta(c)\\
 &=&(((\alpha^{-4}\beta^{2}\omega^{-1}(c_{1})
 \alpha^{-1}(b))\alpha^{-2}\beta\psi^{-1}\omega^{-1}(c_{21}))\alpha(a))\beta\psi^{-2}(c_{22})\\
 &&+(\alpha^{-2}\beta^{2}\omega^{-2}(c_{11})\alpha(a)) ((\alpha^{-2}\beta\psi^{-1} \omega^{-1}(c_{12})\alpha\beta^{-1}(b))\psi^{-1} (c_{2}))\\
 &&+\lambda(\alpha^{-2}\beta^{2}\omega^{-1}(c_{1})\alpha(a))(\alpha(b)\psi^{-1}(c_{2})) +\lambda((\alpha^{-3}\beta^{2}\omega^{-1}(c_{1})b)
 \alpha(a))\psi^{-1}\beta(c_{2})\\
 &\stackrel{(\mref{eq:1.3})(\mref{eq:1.9})}=&(((\alpha^{-4}\beta^{2}\omega^{-1}(c_{1})
 \alpha^{-1}(b))\alpha^{-2}\beta\psi^{-1}\omega^{-1}(c_{21}))\alpha(a))\beta\psi^{-2}(c_{22})\\
 &&+(((\alpha^{-4}\beta^{2}\omega^{-1}(c_{1})\alpha^{-1}(a)) \alpha^{-2}\beta\psi^{-1} \omega^{-1}(c_{21}))\alpha(b))\psi^{-2}\beta (c_{22})\\
 &&+\lambda((\alpha^{-3}\beta^{2}\omega^{-1}(c_{1})a)\alpha(b))\psi^{-1}\beta(c_{2}) +\lambda((\alpha^{-3}\beta^{2}\omega^{-1}(c_{1})b)
 \alpha(a))\psi^{-1}\beta(c_{2}),
 \end{eqnarray*}
 completing the proof since the positions of $a$ and $b$ are symmetric.
 \end{proof}

 We next provide a new way to  construct BiHom-pre-Lie algebra from $\l$-infBH-bialgebra in which the structure maps $\alpha,\beta,\psi,\omega$ are {\bf not} invertible. Here we omit the proof since its proof is similar to Theorem \mref{thm:13.3}.

 \begin{thm}\mlabel{thm:13.10} Let $(A,\mu,\Delta,\alpha,\beta,\psi,\omega)$ be a $\l$-infBH-bialgebra. Then $(A, \star, \alpha^{2}\beta, \alpha^{2}\beta^{2}\psi\omega)$ is a BiHom-pre-Lie algebra, where
 \begin{eqnarray}
 &\star:A\otimes A\longrightarrow A,\ a\star b=(\beta^{2}\psi(b_{1})\alpha(a))\alpha^{2}\beta\omega(b_{2}).&\mlabel{eq:13.6}
 \end{eqnarray}
 \end{thm}

 \begin{rmk} \cite[Theorem 4.6.]{LMMP3} is the special case of Theorem \mref{thm:13.10} for $\l=0$.
 \end{rmk}

 \begin{rmk} From the perspective of the commutative diagram below, Example \ref{ex:12.3} is very interesting.
$$\hspace{-25mm} \xymatrix@C5.5em@R=6ex@W=2ex@H=1em{
&& \text{1(\text{resp.}-1)-infBH-}\atop \text{bialgebras}
\ar@2{->}^{\text{Theorem}\  \ref{thm:13.3}\quad}[r]
\ar@2{<-}^{\l=1 (\text{resp.}-1)\text{ in Example}\  \ref{ex:12.3}\quad}[d]&
\text{BiHom-pre-Lie}\atop \text{algebras}
\ar@2{->}^{\text{\cite[Proposition 3.4]{LMMP6}}}[d]& \\
&&\text{BiHom-associative}\atop \text{algebras}\ar@2{->}^{\quad\text{\cite[Proposition 3.15]{GMMP}}}[r]&\text{BiHom-Lie}\atop \text{algebras}\\
 & }
$$
 \end{rmk}

\subsection{Comodules over (anti-)coquasitriangular $\l$-infBH-bialgebra}\label{se:comodule}
 Coquasitriangular infinitesimal bialgebras were introduced in \cite{MMS} from the mixed bialgebra and its BiHom-version was studied in \cite{MLiY}. In this subsection, we prove that every comodule over (anti-)coquasitriangular $\l$-infBH-bialgebra can induce a $\l$-infBH-Hopf module. Most of the conclusions in this subsection are parallel to those in Section \mref{se:module}, we sketch the proof for the convenience of reading.

 Let $\l$ be a given element in $K$, $(C, \D, \v, \psi, \om)$ be a counitary BiHom-coassociative coalgebra satisfying that $\psi,\omega$ are bijective, $\a, \b: C\lr C$ be two linear maps such that Eqs.(\mref{eq:12.1}), (\mref{eq:12.2}) and (\mref{eq:12.31}) hold, and $\sigma \in (C\otimes C)^{\ast}$ be $\a, \b, \psi, \om$-invariant. Here an element $\sigma \in (C\otimes C)^{\ast}$ is $F$-invariant if $\sigma\circ(F\otimes F)=\sigma$, where $F:C\longrightarrow C$ is a linear map. Define linear map $\mu_{\sigma}~(\hbox{resp.}~\widetilde{\mu_{\sigma}}): C\otimes C\longrightarrow C$ by
 \begin{eqnarray}
 &\mu_{\sigma}(c\otimes d)=\alpha\omega^{-1}(c_{1})\sigma(c_{2},\psi (d))-\sigma(\omega (c),d_{1})\beta\psi^{-1}(d_{2})-\lambda\alpha(c)\varepsilon(d)&\mlabel{eq:01.04}
 \end{eqnarray}
 (resp. \begin{eqnarray}
 &\widetilde{\mu_{\sigma}}(c\otimes d)=\alpha\omega^{-1}(c_{1})\sigma(c_{2},\psi (d))-\sigma(\omega (c),d_{1})\beta\psi^{-1}(d_{2})-\lambda\varepsilon(c)\beta(d).)&\mlabel{eq:01.08}
 \end{eqnarray}

 In what follows, we only prove the case corresponding to $\mu_{\sigma}$, the other case for $\widetilde{\mu_{\sigma}}$ can be checked similarly.

 \begin{lem}\mlabel{lem:12.04} The $\mu_{\sigma}$ (resp. $\widetilde{\mu_{\sigma}}$) defined by Eq.(\mref{eq:01.04}) (resp. Eq.(\mref{eq:01.08})) is a $\lambda$-BiHom-coderivation.
 \end{lem}

 \begin{proof} By $\sigma \in (C\otimes C)^{\ast}$ is $\alpha,\beta,\psi,\omega$-invariant, we know that Eq.(\mref{eq:12.11}) holds. For $a,b \in A$, Eq.(\mref{eq:12.12}) for $\mu_{\s}$ can be proved below.
 \begin{eqnarray*}
 &&\hspace{-20mm}\mu_{\sigma}(\omega (c)\otimes d_{1})\otimes \beta(d_{2})+\alpha(c_{1})\otimes \mu_{\sigma}(c_{2}\otimes \psi(d))+\lambda\alpha\omega(c)\otimes\beta\psi(d)\\
 &\stackrel{(\mref{eq:1.9})(\mref{eq:1.11})}=&\alpha\omega^{-1}(c_{11})\otimes \alpha\omega^{-1}(c_{12})\sigma (\psi(c_{2}),\psi^{2}(d))-\sigma(\omega(c),d_{1})\beta\psi^{-1}(d_{21})\otimes \beta\psi^{-1}(d_{22})\\
 &&-\lambda\alpha(c_{1})\otimes\alpha(c_{2})\varepsilon(d)\\
 &\stackrel{}=&\Delta \circ\mu_{\sigma}(c\otimes d),
 \end{eqnarray*}
 as desired.
 \end{proof}

 \begin{thm}\mlabel{thm:12.05} Let $(C,\D,\psi,\omega)$ be a counitary BiHom-coassociative coalgebra such that $\psi,\omega$ are bijective, $\alpha,\beta: C\rightarrow C$ be linear maps, $\sigma \in (C\otimes C)^{\ast}$ be $\alpha, \beta, \psi, \omega$-invariant and moreover Eqs.(\mref{eq:12.1}), (\mref{eq:12.2}) and (\mref{eq:12.31}) hold. Then the $\lambda$-BiHom-coderivation $\mu_{\sigma}$ (resp.\  $\widetilde{\mu_{\sigma}}$) defined by Eq.(\mref{eq:01.04}) (resp.\  Eq.(\mref{eq:01.08})) is BiHom-associative if and only if for all $c, d, e \in C$,
 \begin{eqnarray}
 &&\alpha^{2}\omega^{-1}(c_{1})(\sigma(\alpha\omega^{-1}(c_{21}),\beta\psi(e))
 \sigma(c_{22},\psi^{2}(d))-\sigma(\omega(c_{2}),\psi(d_{1}))\sigma(d_{2},\psi(e))\nonumber\\
 &&+\sigma(\alpha(c_{2}),\beta(e_{2}))\sigma(\omega(d),e_{1})-\lambda(\hbox{resp.}\ (-\l))\sigma(\alpha(c_{2}),\psi\beta(e))\varepsilon(d))\nonumber\\
 &&=(\sigma(\alpha(c_{1}),\beta(e_{1}))\sigma(c_{2},\psi(d))
 -\sigma(\omega(c),d_{1}) \sigma(\omega(d_{2}),\psi(e_{1}))\mlabel{eq:01.05}\\
 &&+\sigma(\alpha\omega(c),\beta\psi^{-1}(e_{12}))\sigma(\omega^{2}(d),e_{11})
 -\lambda(\hbox{resp.}\ (-\l))\sigma(\alpha\omega(c),\beta(e_{1}))\varepsilon(d))
 \beta^{2}\psi^{-1}(e_{2})\nonumber.
 \end{eqnarray}
 \end{thm}

 \begin{proof} For all $c,d,e\in C$, based on Eqs.(\mref{eq:1.9}) and (\mref{eq:1.11}), one can get
 \begin{eqnarray*}
 \alpha(c)(de)\hspace{-3mm}
 &=&\hspace{-3mm}\alpha^{2}\omega^{-1}(c_{1})\sigma(d_{2},\psi(e))
 \sigma(\alpha(c_{2}),\alpha\omega^{-1}
 \psi(d_{1}))-\alpha\beta\psi^{-1}\omega^{-1}(d_{21})\sigma(\psi^{-1}(d_{22}),\psi(e))
 \sigma(\omega(c),d_{1})\\
 \hspace{-3mm}&&\hspace{-3mm}-\alpha^{2}\omega^{-1}(c_{1})\sigma(\omega(d),e_{1})\sigma(\alpha(c_{2}),\beta(e_{2}))
 +\sigma(\omega(d),\omega^{-1}(e_{11}))
 \sigma(\omega\alpha(c),\beta\psi^{-1}(e_{12}))\beta^{2}\psi^{-1}(e_{2})\\
 \hspace{-3mm}&&\hspace{-3mm}-\lambda\varepsilon(e)\alpha^{2}\omega^{-1}(c_{1})\sigma(c_{2},\psi(d))
 +\lambda\sigma(\omega(c),d_{1})\varepsilon(e)\alpha\beta\psi^{-1}(d_{2})
 +\lambda^{2}\alpha^{2}(c)\varepsilon(d)\varepsilon(e)
 \end{eqnarray*}
 and
 \begin{eqnarray*}
 (cd)\beta(e)
 \hspace{-3mm}&=&\hspace{-3mm}\alpha^{2}\omega^{-1}(c_{1})\sigma(\alpha\omega^{-1}(c_{21}),\psi\beta(e))
 \sigma(\psi^{-1}(c_{22}),\psi(d))-\sigma(c_{2},\psi(d))\sigma(\alpha(c_{1}),\beta(e_{1}))\beta^{2}\psi^{-1}(e_{2})\\
 \hspace{-3mm}&&\hspace{-3mm}-\lambda\alpha^{2}\omega^{-1}(c_{1})\varepsilon(e)\sigma(c_{2},\psi(d))-\alpha\beta\psi^{-1}\omega^{-1}(d_{21})\sigma(\psi^{-1}(d_{22}),\psi(e))
 \sigma(\omega(c),d_{1})\\
 &&+\sigma(\omega(c),d_{1})\sigma(\omega\psi^{-1}(d_{2}),e_{1})\beta^{2}\psi^{-1}(e_{2})+\lambda\alpha\beta\psi^{-1}(d_{2})\varepsilon(e)\sigma(\omega(c),d_{1})\\
 \hspace{-3mm}&&\hspace{-3mm}-\lambda\alpha^{2}\omega^{-1}(c_{1})\varepsilon(d)\sigma(\alpha(c_{2}),\psi\beta(e)) +\lambda\varepsilon(d)\sigma(\omega\alpha(c),\beta(e))\beta^{2}\psi^{-1}(e_{2})
 +\lambda^{2}\alpha^{2}(c)\varepsilon(d)\varepsilon(e).
 \end{eqnarray*}
 Therefore we can finish the proof.
 \end{proof}

 Now we introduce the nonhomogeneous type of coassociative BiHom-Yang-Baxter equation.

 \begin{defi}\mlabel{de:12.03} Let $(C,\D,\v,\psi,\omega)$ be a counitary BiHom-coassociative coalgebra, $\alpha,\beta: C\longrightarrow C$ be linear maps and $\sigma \in (C\otimes C)^{\ast}$. We call
 \begin{eqnarray}
 &\sigma (\alpha(c_{1}),\beta\omega (e))\sigma (c_{2},\psi (d))-\sigma (\omega(c),d_{1})\sigma (d_{2},\psi (e))\qquad\qquad\qquad\qquad\qquad&\nonumber\\
 &\qquad\qquad\qquad+\sigma(\omega(d),e_{1})\sigma (\alpha\psi(c),\beta(e_{2}))=\lambda(\hbox{resp.}\ (-\l))\sigma(\alpha(c),\beta(e))\varepsilon(d)\qquad&\mlabel{eq:01.03}
 \end{eqnarray}
 the {\bf $\lambda$(\hbox{resp.}\ ($-\l$))-coassociative BiHom-Yang-Baxter equation (abbr. $\lambda$(\hbox{resp.}\ ($-\l$))-coabhYBe) in $(C,\D,\v,\psi,\omega)$} where $\lambda$ is a given element in $K$.
 \end{defi}

 By Lemma \mref{lem:12.04} and Theorem \mref{thm:12.05}, we have

 \begin{cor}\mlabel{cor:13.003} Let $(C,\Delta,\varepsilon,\psi,\omega)$ be a counitary BiHom-coassociative coalgebra such that $\psi,\omega$ are bijective, $\alpha,\beta: A\longrightarrow A$ be linear maps, $\sigma\in (C\otimes C)^{\ast}$ be $\alpha,\beta,\psi,\omega$-invariant and moreover Eqs.(\mref{eq:12.1}), (\mref{eq:12.3}) and (\mref{eq:12.31}) hold. If $\sigma$ is a solution of the $\lambda$(\hbox{resp.}\ ($-\l$))-coabhYBe, then $(C, \mu=\mu_{\sigma} (\hbox{resp.}\  \widetilde{\mu_{\sigma}}),\Delta, \varepsilon,\alpha,\beta,\psi,\omega)$ is a $\lambda$-infBH-bialgebra, where $\mu_{\sigma}$(\hbox{resp.}\  $\widetilde{\mu_{\sigma}}$) is defined by Eq.(\mref{eq:01.04})(\hbox{resp.}\  Eq.(\mref{eq:01.08})).
 \end{cor}

 \begin{defi}\mlabel{de:13.004} Under the assumption of Corollary \mref{cor:13.003}, a {\bf coquasitriangular (resp.\ anti-coquasi triangular) counitary $\l$-infBH-bialgebra} is a 8-tuple $(C, \Delta, \varepsilon, \alpha, \beta, \psi, \omega, \sigma)$ consisting of a counitary BiHom-coassociative coalgebra $(C, \Delta, \varepsilon, \psi, \omega)$ and a solution $\sigma\in (C\otimes C)^{\ast}$ of a $\l$(\hbox{resp.}\ $(-\l)$)-coabhYBe.
 \end{defi}

 \begin{pro}\mlabel{pro:12.08} Under the assumption of Corollary \mref{cor:13.003}, $(C, \mu=\mu_{\sigma}(\hbox{resp.}\  \widetilde{\mu_{\sigma}}), \Delta, \varepsilon, \alpha, \beta, \psi, \omega)$, where $\mu_{\sigma}(\hbox{resp.}\  \widetilde{\mu_{\sigma}})$ is defined by Eq.(\mref{eq:01.04})(\hbox{resp.}\  Eq.(\mref{eq:01.08})), is a coquasitriangular (\hbox{resp.}\  anti-coquasitriangular) counitary $\l$-infBH-bialgebra if and only if
 \begin{eqnarray}
 &\sigma(cd,\beta(e))=-\sigma(\omega(d),e_{1})\sigma(\alpha\psi(c),\beta(e_{2}))& \mlabel{eq:01.06}
 \end{eqnarray}
 or
 \begin{eqnarray}
 &\sigma(\alpha(c),de)=\sigma(\alpha(c_{1}),\beta\omega(e))\sigma(c_{2},\psi(d))
  -\lambda\sigma(\alpha(c),\beta(e))\varepsilon(d)-\lambda\sigma(c,d)\varepsilon(e)& \mlabel{eq:01.07}
 \end{eqnarray}
 (\hbox{resp.}
  \begin{eqnarray}
 &\sigma(cd,\beta(e))=-\sigma(\omega(d),e_{1})\sigma(\alpha\psi(c),\beta(e_{2}))
 -\lambda\sigma(\alpha(c),\beta(e))\varepsilon(d)-\lambda\sigma(d,e)\varepsilon(c)& \mlabel{eq:01.10}
 \end{eqnarray}
 or
 \begin{eqnarray}
  &\sigma(\alpha(c),de)=\sigma(\alpha(c_{1}),\beta\omega(e))\sigma(c_{2},\psi(d))& \mlabel{eq:01.11})
 \end{eqnarray}
 hold for all $c, d, e\in C$.
 \end{pro}

 \begin{proof} We only sketch the proof of the coquasitriangular case as follows. By the definition of $\mu$ in Eq.(\mref{eq:01.04}) and invariant condition for $\s$, one easily checks that
 \begin{eqnarray*}
 \sigma(cd,\beta(e))\sigma(\alpha(c_{1}),\beta\omega(e))\sigma(c_{2},\psi(d))
 -\sigma(\omega(c),d_{1})\sigma(d_{2},\psi(e))-\lambda\sigma(\alpha(c),\beta(e))\varepsilon(d)
 \end{eqnarray*}
 and
 \begin{eqnarray*}
 \sigma(\alpha(c),de)=\sigma(\omega(c),d_{1})\sigma(d_{2},\psi(e))
 -\sigma(\omega(d),e_{1})\sigma(\alpha\psi(c), \beta(e_{2}))-\lambda\varepsilon(e)\sigma(c,d).
 \end{eqnarray*}
 The rest is obvious.
 \end{proof}

 The following theorem provides  constructions of $\lambda$-infBH-Hopf modules from  comodules of (anti-)coquasitriangular counitary $\lambda$-infBH-bialgebras.

 \begin{thm}\mlabel{thm:12.013} Let $(C, \Delta, \varepsilon, \alpha_{C}, \beta_{C}, \psi_{C}, \omega_{C}, \sigma)$ be a coquasitriangular (resp.\ anti-coquasitriangular) counitary $\lambda$-infBH-bialgebra and $(M,\rho,\psi_{M},\omega_{M})$ (resp.\ $(M,\widetilde{\rho},\psi_{M},\omega_{M})$) be a left $(C,\Delta,\psi_{C},\omega_{C})$-comodule, $\alpha_{M},\beta_{M}: M\longrightarrow M$ be linear maps such that $\beta_{M}\ci \psi_{M}=\psi_{M}\ci \beta_{M}$, $\rho\ci \beta_{M}=(\beta_{A}\otimes \beta_{M})\ci \rho$ (resp.\ $\widetilde{\rho}\ci \beta_{M}=(\beta_{A}\otimes \beta_{M})\ci \widetilde{\rho}$). Then $(M, \gamma, \rho, \alpha_{M}, \beta_{M}, \psi_{M}, \omega_{M})$ (resp.\ $(M,\widetilde{\gamma},\widetilde{\rho},\alpha_{M},\beta_{M},\psi_{M},\omega_{M})$) becomes a $\lambda$-infBH-Hopf module with an  action $\gamma: C\otimes M\longrightarrow M$ given by
 \begin{eqnarray*}
 &\gamma(c\otimes m):=-\sigma(\omega_{C}(c),m_{-1})\beta_{M}\psi_{M}^{-1}(m_{0})&
 \end{eqnarray*}
 (\hbox{resp.} \begin{eqnarray*}
 &\widetilde{\gamma}(c\otimes m):=-\sigma(\omega_{C}(c),m_{-1})\beta_{M}\psi_{M}^{-1}(m_{0})
 -\lambda\varepsilon(c)\beta_{M}(m) &)
 \end{eqnarray*}
 for all $c\in C, m\in M$.
 \end{thm}

 \begin{proof} We first prove that $(M, \gamma, \alpha_{M}, \beta_{M})$ is a left $(C, \mu, \a_C, \b_C)$-module. For all $c, d\in C$ and $m\in M$, we have
 \begin{eqnarray*}
 \gamma(\mu_{\sigma}\otimes\beta_{M})(c\otimes d\otimes m)
 &\stackrel{(\mref{eq:01.06})}=&\sigma(\omega_{C}^{2}(d),m_{-11})
 \sigma(\alpha_{C}\psi_{C}\omega_{C}(c),\beta_{C}(m_{-12}))
 \beta_{M}^{2}\psi_{M}^{-1}(m_{0})\\
 &=&\gamma(\alpha_{C}\otimes \gamma)(c\otimes d\otimes m).
 \end{eqnarray*}
 Then in the rest we check the compatibility condition for $\lambda$-infBH-Hopf module.
 \begin{eqnarray*}
 &&\hspace{-20mm}\omega_{C}(c)m_{-1}\otimes\beta_{M}(m_{0})+\alpha_{C}(c_{1})\otimes c_{2}\triangleright\psi_{M}(m)+\lambda\alpha_{C}\omega_{C}(c)\otimes\beta_{M}\psi_{M}(m)\\
 &\stackrel{(\mref{eq:01.04})}=&-\sigma(\omega_{C}^{2}(c),\omega_{C}(m_{-1}))\beta_{C}\psi_{C}^{-1}(m_{0-1})\otimes \beta_{M}\psi_{M}^{-1}(m_{00})\\
 &=&(c\triangleright m)_{-1}\otimes (c\triangleright m)_{0},
 \end{eqnarray*}
 completing the proof.
 \end{proof}

\subsection{BiHom-pre-Lie coalgebras from $\l$-infBH-bialgebras}
 In this subsection, we provide two approaches to construct BiHom-pre-Lie coalgebras from $\l$-infBH-bialgebras, one of which can also provide a commutative diagram corresponding to the one in Sec. \ref{se:subprelie}.

 \begin{defi}\mlabel{de:C01.07}(\cite[Definition 3.18]{MLiY}) {\bf A (left) BiHom-pre-Lie coalgebra} is a 4-tuple $(C,\Delta,\psi,\omega)$ where $C$ is a linear space and $\Delta: C \rightarrow C\otimes C$ (write $\D(c)=c_{[1]}\o c_{[2]}$), $\psi, \omega: C \rightarrow C$ are linear maps satisfying
 \begin{eqnarray*}
 &\psi \circ \omega=\omega \circ \psi,~\Delta \circ \psi=(\psi\otimes \psi)\circ\Delta,~ \Delta \circ \omega=(\omega\otimes \omega)\circ\Delta,&\mlabel{eq:C01.013a}\\
 &\bar{\Delta}-\Phi _{(12)}\bar{\Delta}=0,& \mlabel{eq:C01.013}
 \end{eqnarray*}
 where
 $\bar{\Delta}(c)=\omega \psi(c_{[1]}) \otimes \omega(c_{[2][1]}) \otimes c_{[2][2]} - \psi(c_{[1][1]}) \otimes \omega(c_{[1][2]}) \otimes \psi(c_{[2]})$ and $\Phi _{(12)}(a\otimes b \otimes c)=b\otimes a \otimes c$.
 \end{defi}

 \begin{thm}\mlabel{thm:co13.3} Let $(C,\mu,\Delta,\alpha,\beta,\psi,\omega)$ be a $\l$-infBH-bialgebra such that $\alpha,\beta,\psi,\omega$ are invertible. Then $(C,\Delta_{\star},\psi,\omega)$ is a BiHom-pre-Lie coalgebra, where
 \begin{eqnarray*}
 &\Delta_{\star}:C\longrightarrow C\otimes C,\ \Delta_{\star}(c)=\psi^{-1}(c_{12})\otimes\alpha^{-1}\psi\omega^{-2}(c_{11})\beta^{-1}(c_{2}).&\mlabel{eq:co13.2}
 \end{eqnarray*}
 \end{thm}

 \begin{proof} For all $c,d \in C$, we calculate
 \begin{eqnarray*}
 (\omega\otimes \id)\Delta_{\star}(cd)\hspace{-2mm}
 &\stackrel{(\mref{eq:12.2})(\mref{eq:12.4})}=&\hspace{-2mm}
 \beta\psi^{-1}\omega(d_{12})\otimes(\alpha^{-1}\psi(c)\alpha^{-1}\psi\omega^{-2}(d_{11}))d_{2}+\psi^{-1}\omega^{2}(c_{2})\omega(d_{1})\otimes\psi\omega^{-1}(c_{1})d_{2}
 \\
 &&+\lambda\beta\omega(d_{1})\otimes\psi(c)d_{2}+\alpha\psi^{-1}\omega(c_{12})\otimes\psi\omega^{-2}(c_{11})(\beta^{-1}(c_{2})\beta^{-1}\psi(d))\\
 &&+\lambda\alpha\psi^{-1}\omega^{2}(c_{2})\otimes\psi\omega^{-1}(c_{1})\psi(d),
 \end{eqnarray*}
 and
 \begin{eqnarray*}
 (\id\otimes\omega\otimes \id)(\omega\psi\otimes\Delta_{\star})\Delta_{\star}(c)\hspace{-3mm}
 &=&\hspace{-3mm}
 \omega(c_{12})\otimes\psi^{-1}\omega(c_{212})\otimes(\alpha^{-2}\psi^{2}\omega^{-2}(c_{11})
 \alpha^{-1}\beta^{-1}\psi\omega^{-2}(c_{211}))\beta^{-1}(c_{22})\\
 &&\hspace{-3mm}+\omega(c_{12})\otimes\alpha^{-1}(c_{112})\beta^{-1}\omega(c_{21})
 \otimes\alpha^{-1}\psi^{2}\omega^{-3}(c_{111})\beta^{-1}(c_{22})\\
 &&\hspace{-3mm}+\lambda\omega(c_{12})\otimes\omega(c_{21})
 \otimes\alpha^{-1}\psi^{2}\omega^{-2}(c_{11})\beta^{-1}(c_{22})\\
 &&\hspace{-3mm}+\omega(c_{12})\otimes\omega^{-1}(c_{1112})\otimes\alpha^{-1}\psi^{2}\omega^{-4}(c_{1111})
 (\alpha^{-1}\beta^{-1}\psi\omega^{-2}(c_{112})\beta^{-2}\psi(c_{2}))\\
 &&\hspace{-3mm}+\lambda\omega(c_{12})\otimes c_{112}\otimes\alpha^{-1}\psi^{2}\omega^{-3}(c_{111})\beta^{-1}\psi(c_{2}).
 \end{eqnarray*}
 Morever,
 \begin{eqnarray*}
 (\psi\otimes\omega\otimes \id)(\Delta_{\star}\otimes\psi)\Delta_{\star}(c)
 \stackrel{(\mref{eq:1.9})}=\omega(c_{12})
 \otimes\alpha^{-1}(c_{112})\beta^{-1}\omega(c_{21})
 \otimes\alpha^{-1}\psi^{2}\omega^{-3}(c_{111})\beta^{-1}(c_{22}).
 \end{eqnarray*}
 Hence,
 \begin{eqnarray*}
 \bar{\Delta}_{\star}(c)
 \hspace{-3mm}&\stackrel{(\mref{eq:1.3})(\mref{eq:1.9})}=&\hspace{-3mm}\omega(c_{12})\otimes\omega^{-1}(c_{1112})\otimes\alpha^{-1}\psi^{2}\omega^{-4}(c_{1111})
 (\alpha^{-1}\beta^{-1}\psi\omega^{-2}(c_{112})\beta^{-2}\psi(c_{2}))\\
 &&\hspace{-6mm}+\omega^{-1}(c_{1112})\otimes\omega(c_{12})\otimes \alpha^{-1}\psi^{2}\omega^{-4}(c_{1111})
 (\alpha^{-1}\beta^{-1}\psi\omega^{-2}(c_{112})\beta^{-2}\psi(c_{2}))\\
 &&\hspace{-6mm}+\lambda\omega(c_{12})\otimes c_{112}\otimes\alpha^{-1}\psi^{2}\omega^{-3}(c_{111})\beta^{-1}\psi(c_{2})+\lambda c_{112}\otimes\omega(c_{12})
 \otimes\alpha^{-1}\psi^{2}\omega^{-3}(c_{111})\beta^{-1}\psi(c_{2}).
 \end{eqnarray*}
 Thus we have $\bar{\Delta}_{\star}-\Phi _{(12)}\bar{\Delta}_{\star}=0$. The proof is completed.
 \end{proof}

 The following way does not need the condition that the structure maps $\alpha,\beta,\psi,\omega$ are invertible.

 \begin{thm}\mlabel{thm:co13.10} Let $(C,\mu,\Delta,\alpha,\beta,\psi,\omega)$ be a $\l$-infBH-bialgebra. Then $(C,\Delta_{\star}, \psi\omega^{2}, \alpha\beta\psi^{2}\omega^{2})$ is a BiHom-pre-Lie coalgebra, where
 \begin{eqnarray*}
 &\Delta_{\star}:C\longrightarrow C\otimes C,\ \Delta_{\star}(c)=\omega(c_{12})\otimes\beta\psi^{2}(c_{11})\alpha\psi\omega^{2}(c_{2}).&\mlabel{eq:co13.6}
 \end{eqnarray*}
 \end{thm}

 \begin{proof}Similar to Theorem \mref{thm:co13.3}. \end{proof}

 \section{Further research}
 Let $(A, \mu, \Delta, \alpha_{A}, \beta_{A}, \psi_{A}, \omega_{A})$ be a $\lambda$-infBH-bialgebra, $M$ a vector space, $\alpha_{M}$, $\beta_{M}$, $\psi_{M}$, $\omega_{M}:\ M\longrightarrow M$ be four linear maps such that any two of them commute. A {\bf $\lambda$-infBH-Hopf bimodule} over  $(A,\mu,\Delta,\alpha_{A},\beta_{A},\psi_{A},\omega_{A})$ is a 9-tuple $(M,\gamma,\nu,\rho,\varphi,\alpha_{M},\beta_{M},\psi_{M},\omega_{M})$, where
 \begin{eqnarray*}
 \gamma:A\otimes M\longrightarrow M,\ \nu:M\otimes A\longrightarrow M,\ \rho:M\longrightarrow A\otimes M\ \hbox{and}\ \varphi:M\longrightarrow M\otimes A
 \end{eqnarray*}
 are linear maps satisfying the following conditions:

 (1) $(M,\gamma,\rho,\alpha_{M},\beta_{M},\psi_{M},\omega_{M})$ is a left $\lambda$-infBH-Hopf module over $(A,\mu,\Delta,\alpha_{A},\beta_{A},\psi_{A},\omega_{A})$.

 (2) $(M,\nu,\varphi,\alpha_{M},\beta_{M},\psi_{M},\omega_{M})$ is a right $\lambda$-infBH-Hopf module over $(A,\mu,\Delta,\alpha_{A},\beta_{A},\psi_{A},\omega_{A})$.

 (3) $(M,\gamma,\nu,\alpha_{M},\beta_{M})$ is a BiHom-bimodule over $(A,\mu, \alpha_{A},\beta_{A})$.

 (4) $(M,\rho,\varphi,\psi_{M},\omega_{M})$ is a BiHom-bicomodule over $(A, \Delta, \psi_{A},\omega_{A})$.

 (5) the following equations hold:
 \begin{eqnarray}
 &(\gamma\otimes\beta_{A})(\omega_{A}\otimes\varphi)=\varphi\circ\gamma,&\mlabel{eq:20.01}\\
 &(\alpha_{A}\otimes\nu)(\rho\otimes\psi_{A})=\rho\circ\nu.&\mlabel{eq:20.02}
 \end{eqnarray}

 In the forthcoming paper \cite{MM}, motivated by a class of $\l$-infinitesimal BiHom-biproduct bialgebra, we discuss the notion above, which makes the following diagram commutative.

$$ 
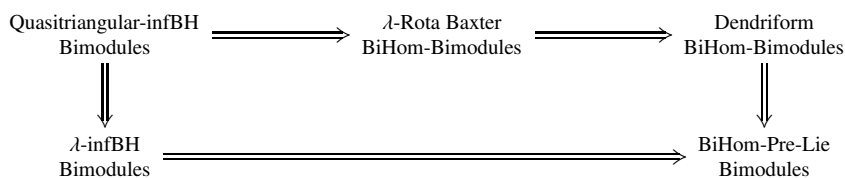
\xymatrix@C4.3em{
&\text{Quasitriangular-infBH}\atop \text{Bimodules}
\ar@2{->}^{ }[r]
\ar@2{->}_{ }[d]
&\text{$\lambda$-Rota Baxter}\atop \text{BiHom-Bimodules}
\ar@2{->}^{ }[r]
&\text{Dendriform}\atop \text{BiHom-Bimodules}\\
&\text{$\lambda$-infBH}\atop \text{Bimodules}
\ar@2{->}^{ }[rr]
&&\text{BiHom-Pre-Lie}\atop \text{Bimodules}
\ar@2{<-}^{ }[u]
&&}
$$

 \section*{Acknowledgment} Ma is supported by Natural Science Foundation of Henan Province (No.212300410365).


\begin{thebibliography}{99} \small

 \mbibitem{Ag99} M. Aguiar, Infinitesimal Hopf algebras, in: New trends in Hopf algebra theory (La Falda,1999), 1--29, {\em Contemp. Math.} {\bf 267}, Amer. Math. Soc., Providence, RI, 2000.


 \mbibitem{Ag00b} M. Aguiar, On the associative analog of Lie bialgebras.  {\em J. Algebra} {\bf 244} (2001), 492--532.

 \mbibitem{Ag04} M. Aguiar, Infinitesimal bialgebras, pre-Lie and dendriform algebras, Hopf algebras, 1-33, {\em Lecture Notes in Pure and Appl. Math.}  {\bf 237}, Dekker, New York, 2004.

 \mbibitem{Br} T. Brzezi\'{n}ski, Rota-Baxter systems, dendriform algebras and covariant bialgebras. {\em J. Algebra} {\bf 460}(2016), 1--25.

 \mbibitem{EF06} K. Ebrahimi-Fard, Rota-Baxter algebras and the Hopf algebra of renormalization, PhD. Thesis, Bonn University, 2006.


 \mbibitem{Fo10} L. Foissy, The infinitesimal Hopf algebra and the operads of planar forests, {\em Int. Math. Res. Not. IMRN }  {\bf 3}  (2010), 395--435.

 \mbibitem{Fo09} L. Foissy, The infinitesimal Hopf algebra and the poset of planar forests, {\em J. Algebraic Combin.} {\bf 30} (2009), 277--309.


 \mbibitem{GW} X. Gao and  X. M. Wang, Infinitesimal unitary Hopf algebras and planar rooted forests. {\em J. Algebraic Combin.} {\bf 49} (2019), 437--460.


 \mbibitem{GMMP} G. Graziani, A. Makhlouf, C. Menini and F. Panaite, BiHom-Associative algebras, BiHom-Lie algebras and BiHom-Bialgebras.  {\em SIGMA} {\bf  11} (2015),086,34 pages.


 \mbibitem{HLS} J. T. Hartwig, D. Larsson and  S. D. Silvestrov, Deformations of Lie algebras using $\sigma$-derivations.  {\em J. Algebra} {\bf  295} (2006), 340--361.

 \mbibitem{JR} S. A. Joni and  G.-C. Rota, Coalgebras and Bialgebras in Combinatorics,  {\em Studies in Applied Mathematics } {\bf 61}(1979): 93-139. Reprinted in Gian-Carlo Rota on Combinatorics: Introductory papers and commentaries (Joseph P. S. Kung, Ed), Birkh\"auser, Boston (1995).

 \mbibitem{LMMP6} L. Liu, A. Makhlouf, C. Menini and  F. Panaite, BiHom-pre-Lie algebras, BiHom-Leibniz algebras and Rota-Baxter operators on BiHom-Lie algebras. {\em Georgian Math. J.} {\bf 28} (2021), 581--594.

 \mbibitem{LMMP1} L. Liu, A. Makhlouf, C. Menini and  F. Panaite, Rota-Baxter operators on BiHom-associative algebras and related structures.  {\em Colloq. Math.} {\bf 161} (2020), 263--294.

 \mbibitem{LMMP3} L. Liu, A. Makhlouf, C. Menini and  F. Panaite, BiHom-Novikov algebras and infinitesimal BiHom-bialgebras.  {\em J. Algebra} {\bf  560} (2020), 1146-1172.


 \mbibitem{LR06} J.-L. Loday and  M. Ronco, On the structure of cofree Hopf algebras, {\em J. Reine Angew. Math.} {\bf 592} (2006), 123--155.

 \mbibitem{MLi} T. Ma and  J. Li, Nonhomogeneous associative Yang-Baxter equations. {\em Bull. Math. Soc. SCI. Math. Roumanie (N.S.)} {\bf 65(113)(1)}(2022), 97--118.

 \mbibitem{MLiY} T. S. Ma, J. Li and  T. Yang, Coquasitriangular infinitesimal BiHom-bialgebras and related structures.  {\em Comm. Algebra } {\bf 49(6)}(2021), 2423--2443.

 \mbibitem{MM} T. S. Ma and  A. Makhlouf, Infinitesimal (BiHom-)bialgebras of any weight (II): Representations. {\em Preprint} (2023).

 \bibitem{MMS} T. S. Ma, A. Makhlouf and  S. Silvestrov, Rota-Baxter cosystems and coquasitriangular mixed bialgebras. {\em J. Algebra Appl.} {\bf 20} (2021), 2150064.

 \mbibitem{MY} T. S. Ma and  H. Y. Yang, Drinfeld double for infinitesimal BiHom-bialgebras. {\em Adv. Appl. Clifford Algebr. } {\bf 30} (2020), Paper No. 42, 22 pp.


 \mbibitem{MYZZ} T. S. Ma, H. Y. Yang, L. Y. Zhang and  H. H. Zheng, Quasitriangular covariant monoidal BiHom-bialgebras, associative monoidal BiHom-Yang-Baxter equations and Rota-Baxter paired monoidal BiHom-modules.  {\em Colloq. Math. } {\bf 161} (2020), 189-221.

 \bibitem{MS}  A. Makhlouf  and S. Silvestrov, Hom-algebra structures, J. Gen. Lie Theory Appl.
2 (2) (2008), 51--64 .

 \mbibitem{OP10} O. Ogievetsky and  T. Popov, R-matrics in rime, Adv. Theor. Math. Phys, 14, (2010), 439--505.


 \mbibitem{WW14} S. X. Wang and  S. H. Wang,  Drinfeld double for braided infinitesimal Hopf algebras, {\em Comm. Algebra} {\bf 42} (2014),  2195--2212.



 \mbibitem{Yau10}D. Yau, Infinitesimal Hom-bialgebras and Hom-lie bialgebras, arXiv: 1001.5000.



 \mbibitem{ZCGL18} Y. Zhang, D. Chen, X. Gao and Y. F. Luo, Weighted infinitesimal unitary bialgebras on rooted forests and weighted cocycles, {\em Pacific J. Math.} {\bf 302} (2019), 741--766.

 \mbibitem{ZG} Y. Zhang and  X. Gao, Weighted infinitesimal bialgebras, arXiv:1810.10790v3 (2020).



 \mbibitem{ZGL} Y. Zhang, X. Gao and  Y. F. Luo, Weighted infinitesimal unitary bialgebras of rooted forests, symmetric cocycles and pre-Lie algebras. {\em J. Algebraic Combin.} {\bf 53} (2021), 771--803.

 \mbibitem{ZGZL} Y. Zhang, X. Gao and  J. W. Zheng, Weighted infinitesimal unitary bialgebras on matrix algebras and weighted associative Yang-Baxter equations, arXiv:1811.00842.


 \end{thebibliography}
 \end{document}